\documentclass{amsart}
\usepackage{amsmath, amsfonts, latexsym,amsthm, amsxtra, amssymb,graphicx,cite}

\newtheorem{theorem}{Theorem}[section]
\newtheorem{proposition}{Proposition}[section]
\newtheorem{cor}{Corollary}[section]

\newtheorem{definition}{Definition}[section]
\newtheorem{lemma}{Lemma}[section]
\newtheorem{remark}{Remark}[section]

\def\CC{{\mathbf C}}
\def\RR{{\mathbf R}}
\def\val{\mbox{val}}
\def\NN{{\mathbf N}}

\def\SS{{\mathbf S}}
\def\BT{\hat{\mathcal B}}
\def\LT{{\mathcal L}}
\def\CT{{\mathcal T}}
\def\CS{{\mathcal S}}

\def\GG{{\mathcal G}}
\def\EE{{\mathcal E}}

\def\PP{{\mathcal P}}
\def\HH{{\mathcal H}}
\def\OO{{\mathcal O}}
\def\xx{<}
\def\dd{>}

\def\bl{{\mathbf \ell}}

\title[Nagumo norms and summability of singular PDEs]{Exponential type Nagumo norms and summability of formal solutions of singular
partial differential equations}

\author{Zhuangchu LUO}
\address{School of Mathematics and Statistics, Wuhan
University, Wuhan 430072, China}
\email{zhuangchu@yahoo.com.cn}

\author{Hua CHEN}
\address{School of Mathematics and Statistics, Wuhan
University, Wuhan 430072, China}
\email{chenhua@whu.edu.cn}

\author{Changgui ZHANG}
\address{ Laboratoire P. Painlev\'e (UMR -- CNRS 8524), UFR
Math., Universit\'e de Lille 1, Cit\'e scientifique, 59655
Villeneuve d'Ascq cedex, France}
\email{zhang@math.univ-lille1.fr}

\begin{document}

\maketitle

\begin{abstract}
In this paper, we study a class of first order nonlinear
degenerated partial differential equations with
singularity at $(t,x)=(0,0)\in \CC^2$. By means of exponential type Nagumo norm approach, Gevrey asymptotic analysis extends to case of holomorphic parameters by a natural way. A sharp condition is
then established to deduce the $k$-summability for the formal solutions. Furthermore, analytical solutions in conical domains are found for each type of these nonlinear singular PDEs.
\end{abstract}

\begin{abstract}[R\'esum\'e]
 bb
\end{abstract}

\tableofcontents

\numberwithin{equation}{section}
\section{Introduction}\label{section:introduction}

As early as in 1913, Gevrey \cite{Ge} studied following
forward-backward diffusion equations
\begin{equation}\label{equation:Gevrey}
A(t,x)u_x-B(t,x)u_{tt}+(\mbox{lower order terms})=f,
\end{equation}
where the coefficient $A(t,x)$ changes sign through the line
$A(t,x)=0$. Later, this kind of equations has been used widely, e.g.
to deal with the so-called ``counter-current convection diffusion''
process which appeared from some physical or chemical problems. Even
for the simplest forms of the degenerated equation
\eqref{equation:Gevrey}, such as
$$
xu_x-\frac{1}{2} u_{tt}+tu_t=0, \ \  tu_x-u_{tt}=0,
$$
and
$$
x^3u_x-x^2u_{tt}-tu_t=f(t,x),
$$
we can also find some interesting applications in kinetic theory and
stochastic processes (cf. \cite{HO,PT} and references
therein), these examples would be covered by more general
degenerated PDEs, such as
\begin{equation}\label{equation:m}
(t\partial _t)^m
u=F(t,x,(t\partial_t)^i\partial_x^ju),
\end{equation}
where one may assume the indices $i$, $j$ to be such that $in+jm\le
mn$ and $ i<m$, with some positive integers $m$ and $n$.
Note also that several reaction-diffusion equations \cite{FM} of
type
\begin{equation*}\label{equation:reaction-diffusion}
\partial_tu=\partial_{xx}u+f(u)
\end{equation*}  can be written in a form of the \eqref{equation:m} while the time
variable $t$ is put into a suitable ``exponential scale" $t\mapsto
\tau=e^{\lambda t}$.

In this paper, only the case of $m=1$, $n=1$ will be considered for the partial differential equation \eqref{equation:m} under the initial condition $u(0,x)=0$ and the
approach used in the following can be expected to be applied to
general cases. More precisely, we will suppose that $F(t,x,u,v)$ be a function holomorphic
at ${\bf 0}\in\CC^4$ such that $F(0,x,0,0)\equiv 0$. Then, equation
\eqref{equation:m} can be written into the following form:
\begin{eqnarray}\label{equation:1}
&&t\partial_t u=a(x)t+b(x)u+\gamma (x)\partial_xu+\\
&&\qquad\qquad\qquad\qquad\sum_{i+j+\alpha\ge
2}a_{i,j,\alpha}(x)t^iu^j(\partial_xu)^\alpha,\qquad u(0,x)=0,\nonumber
\end{eqnarray}
where $a(x)$, $b(x)$, $\gamma(x)$, $a_{i,j,\alpha}(x)$ are holomorphic on an
open disc centered at $0\in\CC$.

The existence and uniqueness of
holomorphic solution of \eqref{equation:1} depend mainly on the
valuation of the function $\gamma$ at $x=0$ (see \cite{GT}, chapters
5, 6 and \cite{CLT}). So,
let $p=\val(\gamma)$ be the valuation of $\gamma(x)$ at $x=0$. For the case $p=1$, the existence
and uniqueness of holomorphic solutions of (\ref{equation:1}) are proved in \cite{CT1,CT2,CLH}. For
the case $2\le p<\infty$, if the following condition $(F)$ is satisfied:
$$
b(0)\notin \NN^*=\{1,2,3,...\}\quad\hbox{and}\quad a_{i,j,\alpha}(0)=0,\ \forall\ \alpha>0,\leqno(F)
$$
then, thanks to Theorem 1.2 of \cite{CLT}, the equation \eqref{equation:1} has a unique power
series solution, which is convergent in $t$ and divergent in $x$ with Gevrey order $1/k$ or $1+1/k$ according to convention of \cite{CLT} ($k=p-1$).

\subsection{Main results}\label{subsection:mainresults}

For convenience, we rewrite $\gamma(x)$ as $x^{k+1} c(x)$ and let $c=c(0)$, $b=b(0)$, with $c\not=0$ . One main result of this paper may be the following

\begin{theorem}\label{theorem:yes}
Under the condition $(F)$, the equation \eqref{equation:1}
 has a unique formal solution $\hat u(t,x)$, which is convergent in $t$ and $k$-summable in all
 directions of
the $x$-plane except at most a countable directions belonging to the following set:
\begin{equation}\label{equation:SDbck}
SD_{b,c;k}:=\bigcup_{\nu=0}^{k-1}\bigg\{\frac{\arg({z})+{2\nu\pi}}k:
z\in\big\{\frac 1c, \frac{1-b}c,\frac{2-b}c,\frac{3-b}c,\cdots\big\}\bigg\}.
\end{equation}
\end{theorem}

On the other hand, if the condition $(F)$ is not satisfied,  the
formal power series solution may be divergent in both variables $t$
and $x$. For example, the following nonlinear partial differential equation
\begin{equation}\label{equation:no}
t\partial_t u=a(x)t+x^2\partial_x u+t(\partial_x u)^2,\qquad
u(0,x)=0
\end{equation}
has a unique formal solution in the Gevrey type power series space $\CC[[t,x]]_{\frac
12, 1}$ if $a(x)\not\equiv a(0)$ and $\val(a(x))\leq 1$ (see \cite{CLT}).

\begin{theorem}\label{theorem:no}
Consider the equation (\ref{equation:no}) and suppose that
$a(x)\not\equiv a(0)$ and $\val(a)=0$ or $1$. Let
$$\hat u(t,x)=\sum_{n\ge
0}v_n(x)t^{n+1}$$ be the formal solution of (\ref{equation:no}) and denote 
$$\hat U(\tau,x)=\sum_{n\ge
0}\frac{v_n(x)}{\Gamma(\frac{ n+1}2)}\tau^n$$ 
as the formal $2$-Borel transform of $\hat u(t,x)$ 
on $t$. Then the power series $\hat U$ is convergent in $\tau$ and
Borel summable with respect to the variable $x$ in any direction
excepted in $\RR^+$.
\end{theorem}

However, by using transformation such as $w(t,x)=u(tx,x)$, the condition $(F)$ would be always satisfied for every equation \eqref{equation:1}, provided the initial equation admits a formal solution, e.g. if $b(0)\notin\NN^*$. Applying Theorem~\ref{theorem:yes} to this new equation yields the following result.

\begin{theorem}\label{theorem:all}
For any equation of the form \eqref{equation:1}, if $b(0)\notin\NN^*$ and $\val(\gamma)=k+1$, then for almost every sector $V$ of openness larger than but enough close to $\pi/k$, there exists $R>0$ such that \eqref{equation:1} admits an analytic solution in the associated conical domain $\{(t,x)\in\CC\times V: |t|<R|x|<R^2\}$.
\end{theorem}

The result stated in Theorem~\ref{theorem:yes} is more general than that given in our previous note \cite{LCZ}  where, instead of the condition
$(F)$, the following more restrictive condition is assumed:
$$
b(0)\notin \NN^*=\{1,2,3,...\}\quad\hbox{and}\quad \val(a_{i,j,\alpha})+\nu j\ge \val(\gamma),\ \forall\ \alpha>0,\leqno(F1)
$$
where $\nu=\min(\val(a),\nu_0)$ with $\nu_0=\min\left\{\val(a_{i,0,0}):{i\ge 2}\right\}$;
see \cite{LCZ-2} for more details. In spite of the above condition $(F1)$, we are led to study a convolution PDE that can be decomposed into an infinite dimensional system of nonlinear convolution differential equations. In order to prove the existence of solutions with exponential growth at infinity, we introduce a family of Nagumo type norms to Banach spaces which were used in our previous paper  \cite{LCZ-2}.

The original definition of the $k$-summability may be found in \cite{Ra}; see also \cite{Ba}, where the $k$-summability and the multi-summability are both applied to the analytic ODEs with singularities. Even the situation seems somewhat similar as what happens for singular perturbation problems \cite{CRSS}, the principal framework in our study remains inside the $k$-summability with holomorphic parameters, such as in \cite{MR1}. A more precise version of Theorems~\ref{theorem:yes},~\ref{theorem:no} and~\ref{theorem:all} will be given as Theorems~\ref{theorem:yes1},~\ref{theorem:proofnow} and~\ref{theorem:withoutFgeneral}, respectively, and also by expression \eqref{equation:yes} and Corollary~\ref{rethgen}.

Observe Theorem~\ref{theorem:yes} can be improved to the case where coefficients $a(x)$, ..., $a_{i,j,\alpha}(x)$ of \eqref{equation:1} are only assumed to be $k$-summable in suitable directions; see Theorem~\ref{theorem:yeskcoeff}. In the semilinear case, a simple analytic change of coordinates suffices to resolve any equation by $k$-summable functions (cf. Theorem~\ref{theorem:withoutFlinear}).

\subsection{Plan of the paper}
This paper contains two parts: the part 1, from Section~\ref{section:spaces} to Section~\ref{section:ksummable}, is devoted to a reformulation of $k$-summability with holomorphic parameters by means of Nagumo norms in (generalized) Borel-plane; the part 2 is concentrated to application of results of Part 1 to the class of PDEs of the form~\eqref{equation:1}.

In Section~\ref{section:spaces}, several functional
spaces are introduced by means of a family of exponential-Nagumo
type norms; these spaces may be of interest in a general setting for
studying PDE summability. In Section~\ref{section:lemma}, the main result is Lemma~\ref{lemma:key}, which allows us to give estimates on derivatives of a function in terms of exponential-Nagumo norms; see also Corollary~\ref{cor:key}. Results of these two sections will be extended to any positive level $k>0$ in Section~\ref{section:extension}.

In Section~\ref{section:ksummable}, we will start by recalling some basic definitions or facts on $k$-summability over $\CC$ and therefore deal with a version with holomorphic parameters introduced by J. Martinet and J.-P. Ramis in \cite{MR1}. The Nagumo type norms examined in the previous sections are used and useful as test tool for studying these functions in (generalized) Borel plane.

From Section~\ref{section:conditionF}, we consider equation \eqref{equation:1} and, firstly, by assuming the condition $(F)$ we check an analytical equivalent form for that applying Borel transform gives raise to a {\it good} convolution equation. In Section~\ref{section:proofk=1}, we will  give  the
proof of Theorems~\ref{theorem:yes} for the case of $k=1$, which corresponds exactly to the Borel-summability case. A complete proof of Theorems~\ref{theorem:yes} will be given in Section~\ref{section:proofgeneralcase}.

In Section \ref{section:withoutF}, we consider more general cases in which the condition $(F)$ will be not satisfied. By using some elementary transformations on the initial variables, we study the  summability of the formal solutions in this case, particularly, it will be proved that, in this special case, the equation \eqref{equation:1} admits an analytical solution in some suitable conical domains for each time while the formal solution exists; see Theorem~\ref{theorem:withoutFgeneral} and its Corollary~\ref{rethgen}.

Finally, Theorem~\ref{theorem:no} will be proved in Section~\ref{section:proofno}, together with Theorem~\ref{theorem:proofnow}.

\subsection{Notations and related problems}\label{subsection:notations} The following notations will be used in this paper.
\begin{itemize}
\item For $R>0$ and $a\in\CC$, $D(a;R)$ denotes the open disc $\{\vert x-a\vert<R\}$ in complex plane.
 \item The symbol $\log$ will denote the principal branch of the complex logarithm given over its Riemann surface denoted by $\tilde\CC^*$.
\item The set of non-zero complex numbers can be identified as $(0,\infty)\times \SS^1$, where $\SS^1$ denotes the unit circle $\{\vert x\vert=1\}$. We will call {\it direction (over $\CC$)} any element $d\in \SS^1$, that can be represented by a real number belonging to $[0,2\pi)$.
\item If $\Omega$ denotes a domain of $\CC$ or $\CC^m$ for any positive integer $m$, $\OO(\Omega)$ will be the set of functions defined and analytic in $\Omega$.
    \item For all $k>0$, $\CC[[x]]_{1/k}$ denotes the space of power series of Gevrey order $k$: $\sum_{n\ge 0}a_nx^n\in\CC[[x]]_{1/k}$ if, and only if, $\sum_{n\ge 0}\frac{a_n}{\Gamma(1+n/k)}x^n$ admits a positive radius of convergence. When $k=\infty$, by convention $\CC[[x]]_0=\CC\{x\}$ denotes the set of germs of analytic functions at $x=0$.
\end{itemize}

It would be interesting if results of this paper might be extended and applied to classical equations mentioned in the beginning of Introduction. Also it seems that a generalization to high order equations would be possible whilst $k$-summability with holomorphic parameters would be replaced by multisummability version. In addition, analyzing Stokes phenomenon would be possible and interesting at least for some particular cases, e.g. one of the cases may be the equations of semilinear case.

Since the work \cite{LMS} on the summability of formal solutions of the heat equation, many authors have studied the (multi-)summability for PDEs, see, for example,
\cite{Ba1,CT,Hibino1,Hibino2,Hibino3,Ouchi1,Ouchi2,Ouchi3} and the references therein. Theorem~\ref{theorem:no} of this paper illustrates in what manner a combination of summations in two variables becomes necessary for some singular PDEs.  This study will be continued in a forthcoming work \cite{LZ} while the Gevrey type asymptotic analysis and summability involving two complex variables are considered.

\bigskip

\bigskip

\part[50pt]{Nagumo norms and $k$-summable functions}

A power series is said Borel-summable in a given direction $d$ if its Borel transform represents an analytic function at the origin in the Borel plane, saying $\xi=0$,  which can be analytically extended into a function possessing at most an exponential growth of the first order at the infinity over an open sector bisected by $d$. It is natural to introduce {\it exponential type} norms for functions in the $\xi$-plane.

As it is easy to be seen, any analytic partial differential equation may be, in most of cases, read as an infinte dimensional system of equations while expending along one variable. So one may be led to study a sequence of exponential norms and this is why we will consider Nagumo type norms to improve exponential norms over a sector; see  Section~\ref{section:spaces}. The classical Nagumo's norm (cf. \cite{Na}) consists of some functional norm depending on the distance to the boundary (e.g. a circle for a disc) of every point in a domain where one has to make functional estimates. See \cite[\S 3]{CRSS} and references therein for more information on Nagumo type norms and their applications.

In \S\ref{section:lemma}, Lemma~\ref{lemma:key} will be established for assuming estimates of derivatives in terms of norms of given function; it will play a key role in the proof of Theorem~\ref{theorem:yes}, done in Sections~\ref{section:proofk=1} and~\ref{section:proofgeneralcase} of Part 2.
In \S\ref{section:extension}, after considering extension to the case of a sector joined by a disc -- this is really the case for the classical definition of Borel-summability, we give also $k$-summability version of previous results.

Section~\ref{section:ksummable} is devoted to $k$-summability with holomorphic parameters, inspired by the work \cite{MR1} of J. Martinet and J.-P. Ramis. In terms of Nagumo norms, some equivalent conditions will be given, in Theorem~\ref{theorem:ksummableexpansion}, to assume holomorphic parameters $k$-summability. These creteria will be followed through all of the Part 2 for the study of summability of partial differential equations.

\section{Nagumo norms and some functional spaces}\label{section:spaces}

 Let us start by the following  Banach space $\EE_{S,\mu}$ studied  in \cite{CT} and \cite{LCZ}. For any
$d\in\SS^1$ and $\theta\in(0,\pi)$, we set
$$S(d,\theta)=\{\xi\in\CC^* : \vert\arg\xi-d\vert\xx \theta\}.
$$
 Let
$S=S(d,\theta)$ and $\mu\dd0$;
a functions $f\in\OO(S)$ belongs to $\EE_{S,\mu}$ if
$$
\Vert f\Vert_{S,\mu}:=M_0\sup_{\xi\in S}\vert
f(\xi)(1+\vert\xi\vert^2)e^{-\mu\vert\xi\vert}\vert\xx\infty,
$$
where $M_0$ is the constant given by the formula
\begin{equation}\label{equation:M_0}
M_0=\sup_{s>0}\frac{2(1+s^2)}{s(4+s^2)}\,(\ln(1+s^2)+s\arctan s).
\end{equation}
Among interesting proprieties of $\EE_{S,\mu}$, we are content to
notice that $(\EE_{S,\mu},\Vert\ \Vert_{S,\mu})$ constitutes a
Banach algebra with respect to the convolution product and,
moreover, if $\mu_2\dd\mu_1$  and
$f_i\in\EE_{S,\mu_i}$, then
\begin{equation}\label{equation:f1f2mu2}
 \Vert
f_1*f_2\Vert_{S,\mu_2}\le4[M_0(\mu_2-\mu_1)]^{-1}\Vert
f_1\Vert_{S,\mu_1}\Vert f_2\Vert_{S,\mu_2}\,.
\end{equation}
When $\mu_1=\mu_2$, the above relation \eqref{equation:f1f2mu2} can be modified as follows:
\begin{equation}\label{equation:f1f2mu}
\Vert
f_1*f_2\Vert_{S,\mu_2}
\le \Vert f_1\Vert_{S,\mu_1}\Vert f_2\Vert_{S,\mu_2}.
\end{equation}

Now we introduce some Nagumo type norms for extending these functional spaces. We will see that such norms allow to estimate the
derivatives in terms of any given function; see Section~\ref{section:lemma}, Lemma~\ref{lemma:key} and Corollary~\ref{cor:key}.

\begin{definition}\label{definition:Nagumo} Let $\theta \in(0,\pi)$, $S:=S(d,\theta)$ and let $\mu\in(0,\infty e^{-id})$, {\it i.e} $\mu e^{id}\in(0+\infty)$; for any $\xi\in S$, let
\begin{equation}\label{equation:deltaxi}
\delta(\xi)=\delta(\xi,S):=\min\{d+\theta-\arg\xi,-d+\theta+\arg\xi,1\}\,.
\end{equation}
For any $f\in\OO(S)$ and $n\ge0$, we define:
$$
\|f\|_{S,\mu ,n}:=M_0\sup_{\xi\in S}\left|f(\xi)e^{-\mu
\xi}(1+|\xi|^2)\delta(\xi)^{n} \right|,
$$
where $M_0$ is the positive constant given by \eqref{equation:M_0}.

The function $f$ will be said belonging to $\EE_{S,\mu ,n}$ if  $\|f\|_{S,\mu ,n}<\infty$.
\end{definition}

In the definition \ref{definition:Nagumo}, the parameter $n\geq 0$ can be often chosen as a non-negative integer.

\begin{remark}\label{remark:spaces1}
In Definition~\ref{definition:Nagumo}, contrary to what done in our previous paper \cite{LCZ-2}, we make use of $e^{-\mu \xi}$ instead of $e^{-\mu\vert\xi\vert}$; this modification permits much flexibility to carry arguments inside Complex Analysis. See Corollary~\ref{cor:key}, Proposition~\ref{proposition:NagumoR} and so on.
\end{remark}

We notice firstly that if $\theta<\pi/2$, $S=S(d,\theta)$ and $\mu=\vert\mu\vert e^{-id}$, then the following inclusions hold for any $n\ge 0$:
\begin{equation}\label{equation:3EE}
\EE_{S,\vert\mu\vert\cos\theta}\subset\EE_{S,\mu,0}\subset\EE_{S,\mu,n}.
\end{equation}
Indeed, in view of the fact that $\delta(\xi,S)\le 1$ and
$$\vert\mu\vert\vert\xi\vert\cos\theta<\Re(\mu\xi)\le \vert\mu\vert\vert\xi\vert,\qquad
\forall\ \xi\in S(d, \theta),
$$
it follows that, for any given $f\in\OO(S)$:
\begin{equation}\label{equation:3norms}
\Vert f\Vert_{S,\mu,n}\le \Vert f\Vert_{S,\mu,0}\le \Vert f\Vert_{S,\vert\mu\vert\cos\theta}\,.
\end{equation}

One can easily prove that each $(\EE_{S,\mu,n},\|\cdot\|_{S,\mu,n})$ constitutes a Banach
space. Let $\mu$, $\mu^\prime\in(0,\infty e^{-id})$ with 
$|\mu|\ge |\mu^\prime|$
 and let $n\ge n^\prime\ge0 $. Observe, as in \eqref{equation:3EE} and \eqref{equation:3norms}, the Banach space $\EE_{S,\mu^\prime ,n^\prime}$ can be considered as a
subspace of $\EE_{S,\mu,n}$ and  the following inequality holds:
\begin{equation}\label{equation:2norms}
\forall \ f\in
\EE_{S,\mu^\prime,n^\prime},\qquad
\|f\|_{S,\mu ,n}\le\|f\|_{S,\mu ^\prime,n^\prime}.
\end{equation}

With regard to the estimates of \eqref{equation:f1f2mu2} and \eqref{equation:f1f2mu} for the convolution product, one has following result.

\begin{proposition}\label{proposition:convolution}
Let $S=S(d,\theta)$ and $\mu$ as in Definition~\ref{definition:Nagumo} and let $n$, $n'\ge 0$. The following assertions hold.
\begin{enumerate}
\item \label{assertion:fgmunn'}If $f\in\EE_{S, \mu ,n}$ and $g\in\EE_{S, \mu ,n^\prime}$, then
 $f*g\in\EE_{S, \mu ,n+n^\prime}$
 and
\begin{equation}\label{equation:fgmunn'}
\|f*g\|_{S, \mu ,n+n^\prime}\le \|f\|_{S, \mu ,n}\|g\|_{S, \mu ,n^\prime}
\end{equation}
\item \label{assertion:fgmumu'} Let $\mu^\prime\in(0,\infty e^{-id})$ such that
$|\mu|\le |\mu^\prime|$. If $f\in\EE_{S, \mu ,0}$, $g\in\EE_{S, \mu
^\prime,n}$, then $f*g\in \EE_{S, \mu ^\prime,n}$ and
\begin{equation}\label{equation:fgmumu'}
\|f*g\|_{S, \mu ^\prime,n}\le C_{\mu'-\mu}\,\Vert f\Vert_{S,\mu,0}\,\Vert g\Vert_{S,\mu',n},
\end{equation}
where, $M_0$ being defined by \eqref{equation:M_0}, we set:
$$
C_{\mu'-\mu}=
\frac{4}{M_0\cos(\theta/2)\,|\mu
^\prime-\mu |}.$$
\end{enumerate}
\end{proposition}

\begin{proof}
Let $f\in\EE_{S, \mu ,n}$, $g\in\EE_{S, \mu ,n^\prime}$ and let $\xi\in S$. For any $\tau\in(0,\xi)$, it follows that $\delta(\tau)=\delta(\xi-\tau)=\delta(\xi)$; hence,  the following inequality holds for all $\tau\in(0,\xi)$:
\begin{equation}\label{equation:fgtau}
\vert f(\tau)\,g(\xi-\tau)\vert\le C_{f,g}\,\frac{\vert e^{\mu\xi}\vert\,\delta(\xi)^{-n-n'}}{(1+\vert \xi-\tau\vert^2)(1+\vert\tau\vert^2)}\,,
\end{equation}
where we set
$$
C_{f,g}:=\frac{\Vert f\Vert_{S,\mu,n}\,\Vert g\Vert_{S,\mu,n'}}{M_0\,^2}\,.
$$

By expressing $f*g(\xi)$ as integral of $\tau\mapsto f(\tau)\,g(\xi-\tau)$ over interval $(0,\xi)$ and by considering \eqref{equation:fgtau} in this expression, one can deduce that
$$
|f*g(\xi)|
\le C_{f,g}\,\delta(\xi)^{-n-n^\prime}\,\vert e^{\mu\xi}\vert\Big\vert\int^{\xi}_0
\frac{d\tau}
{(1+|\xi-\tau\vert^2)(1+\vert\tau\vert^2)}\Big\vert\,.
$$
If we define
$$
I(s)=\int_0^s\frac{dt}{(1+(s-t)^2)(1+t^2)}\quad\forall s>0,
$$
then we get the following estimate:
$$
|f*g(\xi)|
\le C_{f,g}\,\delta(\xi)^{-n-n^\prime}\,\vert e^{\mu\xi}\vert\,I(\vert\xi\vert).
$$
Since
$$
I(s)=\frac{2}{s(4+s^2)}\,(s\arctan s+\ln(1+s^2))\le \frac{M_0}{1+s^2},
$$
we obtain the estimate \eqref{equation:fgmunn'}, which implies that $f*g\in\EE_{S,\mu,n+n'}$, the first part of Proposition~\ref{proposition:convolution} is proved.

Next, let $f\in\EE_{S, \mu ,0}$, $g\in\EE_{S, \mu^\prime,n}$, instead of \eqref{equation:fgtau}, we have 
\begin{equation}\label{equation:fgtau1}
\vert f(\tau)\,g(\xi-\tau)\vert\le C'_{f,g}\,\frac{\vert e^{\mu'\xi-(\mu'-\mu)\tau}\vert\,\delta(\xi)^{-n}}{(1+\vert \xi-\tau\vert^2)(1+\vert\tau\vert^2)}\,,
\end{equation}
where $C'_{f,g}$ is a similar constant as $C_{f,g}$, thus by similar way, we can prove the estimate \eqref{equation:fgmumu'} holds, the second part of  Proposition~\ref{proposition:convolution} is proved.
\end{proof}

If we take $n=n'=0$ in \eqref{equation:fgmunn'}, we find following corollary.
\begin{cor}\label{cor:nn'0}
The Banach space $\EE_{S,\mu,0}$ constitutes a Banach algebra w.r.t. the convolution product.
\end{cor}

\begin{proof}
 It is clear.
\end{proof}

On the other hand, from Proposition~\ref{proposition:convolution}, one can not know whether the space $(\EE_{S,\mu,n},\|\cdot\|_{S,\mu,n})$ does constitute a Banach algebra w.r.t. the convolution product when $n\ge 1$.

\section{A key lemma}\label{section:lemma}

In this section, the main result is Lemma~\ref{lemma:key}, in which we will give an estimate of the first order derivative of a function in functional spaces introduced in Section~\ref{section:spaces}. Let $SD_{b,c;k}$ be the set  given by \eqref{equation:SDbck}. It is easy to check that, for any direction $d$ which does not belong to
$SD_{b,c;k}$, there exist positive constants $\theta$ and $\sigma$ such that
for any $n\in\NN^*$ and $\xi\in S(d,\theta)$, the following estimate holds:
\begin{equation}\label{equation:nbcsigma}
|n-b-c\xi^k|\ge \sigma(n+|\xi^k|),
\end{equation}
where $b=b(0)$ and $c=c(0).$

\begin{lemma}[Key Lemma]\label{lemma:key} Let $\theta \in(0,\pi)$, $S:=S(d,\theta)$ and $n$ be a positive integer. If for $k=1$ and $\sigma>0$ the 
inequality (\ref{equation:nbcsigma}) holds and $(n-b-c\xi)f\in\EE_{S, \mu ,n-1}$,
then $\xi\partial_\xi f\in\EE_{S, \mu ,n}$ and
\begin{equation}\label{equation:key}\|\xi\partial_\xi
f\|_{S, \mu ,n}\le E \|(n-b-c\xi)f\|_{S, \mu ,n-1},
\end{equation} where
$E=\sigma^{-1}(e^3+|\mu |)$ and $\sigma$ is a positive constant satisfying (\ref{equation:nbcsigma}) in the case of $k=1$.
\end{lemma}

The proof of Lemma~\ref{lemma:key} will be given later in this section, which will depend on following  two propositions.

\subsection{Nagumo norms inside Cauchy formula}\label{subsection:lemmaCauchy}  Notice that the function $\xi\mapsto\delta(\xi)$ given by \eqref{equation:deltaxi} depends on the angular distance of $\xi$ to the half-lines sides of the sector $S$. If we set $\eta=\log\xi$, this means that $\xi=e^\eta$, then the sector $S=S(d,\theta)$ will be transformed into a horizontal strip $\Omega:=\Omega(d,\theta)$, which can be identified to the unbounded rectangular domain $\RR\times (d-\theta,d+\theta)i$. Let $d(\zeta)=\delta(e^\zeta)$ for any $\zeta\in \Omega$, and let $d(\xi,\partial \Omega)$ be the distance from $\zeta$ to the boundary of $\Omega$. It follows:
\begin{equation}\label{equation:dboundary}
d(\zeta)\le \min\{1,d(\zeta,\partial \Omega)\},\qquad
d(\zeta+\eta)\ge d(\zeta)-\vert \eta\vert
\end{equation}
for any $(\zeta,\eta)\in \Omega\times \Omega$ such that $\zeta+\eta\in \Omega$.

\begin{proposition}\label{proposition:Dff'}
Let $D$ be a simply connected region of the complex plane and let $d(\zeta)$ be a positive function defined in $D$ satisfying the condition \eqref{equation:dboundary} where $\Omega$ is replaced by $D$.
Let  $f\in\OO(D)$.
If there exist $k>0$, $n\ge0$ and $C>0$ such that, for any $\zeta\in D$,
\begin{equation}\label{equation:fCkn}
|f(\zeta)|\le\frac C{(1+|e^\zeta|^{2k})d(\zeta)^n},
\end{equation}
then the following estimate holds over the whole domain $D$:
\begin{equation}\label{estd}
\left|f^\prime(\zeta) \right|\le
\frac{e^{2k+1}(n+1)C}{(1+|e^\zeta|^{2k})d(\zeta)^{n+1}}.
\end{equation}
\end{proposition}
\begin{proof} Let $\zeta\in D$ and choose a positive $r$ such that $r<d(\zeta)$. Let $C_{\zeta,r}$ be the positively oriented circle centered at $\zeta$ with radius $r$. By using Cauchy formula, it follows:
$$
f'(\zeta)=\frac1{2\pi i}\,\int_{C_{\zeta,r}}\frac
{f(y)}
{(y-\zeta)^2}\,dy.$$
Replacing $\zeta$ by $\zeta+re^{i\alpha}$ in \eqref{equation:fCkn}, one has
$$
|f'(\zeta)|\le \frac 1{2\pi r}\int_0^{2\pi}\frac {C\,d\alpha}
{(1+|e^{\zeta+re^{\alpha i}}|^{2k})[d(\zeta+re^{\alpha i})]^n },
$$
which implies that
\begin{equation}
\vert f'(\zeta)|\le \frac {e^{2k}C} {1+|e^\zeta|^{2k}}\,\frac 1{r
[d(\zeta)-r]^n},\label{equation:f'r}
\end{equation}
in view of \eqref{equation:dboundary} and of the fact that $r<d(\zeta)\le1$.

If $n=0$, from \eqref{equation:f'r} we get the required estimate \eqref{estd}  by choosing $r=\frac{ d(\zeta)}e$. If $n\ge 1$, we choose $r=\frac{d(\zeta)}{n+1}$, which implies the estimate \eqref{estd} from \eqref{equation:f'r}; indeed, we have following obvious estimate:
$$
\frac{1}{r[d(\zeta)-r]^n}= \frac{n+1}{d(\zeta)}
\left(\frac{1}{d(\zeta)}\frac{1+n}{n}\right)^n =
\frac{n+1}{d(\zeta)^{n+1}}\Big(\frac{n+1}n\Big)^n\le
 \frac{e(n+1)}{d(\zeta)^{n+1}}.
$$
Proposition~\ref{proposition:Dff'} is then proved.
\end{proof}

The following result can be proved as a direct application of Proposition~\ref{proposition:Dff'} with $D=\Omega=\log S(d,\theta)$ and $d(\zeta)=\delta(e^\zeta)$.
\begin{proposition}\label{proposition:Shh'}
Let  $h\in\OO(S)$ with $S=S(d,\theta)$. Let $\delta(\xi)$ be as in \eqref{equation:deltaxi}. If there exist constants $k>0$, $n\ge 0$ and $C>0$ such that
$$
|h(\xi)(1+|\xi|^{2k})\delta(\xi)^n|\le C
$$
for all $\xi\in S$,
then the following estimate holds over $S$:
\begin{equation}\label{equation:hnkC}
 |(1+|\xi|^{2k})\delta(\xi)^{n+1}\xi \partial_\xi h(\xi)|\le (n+1) e^{2k+1}
C.
\end{equation}
\end{proposition}

\begin{proof}
It suffices to apply Proposition~\ref{proposition:Dff'} to the function
$f(\zeta)=h(e^\zeta)$ for $\zeta\in \Omega=\log(S)$ and $d(\zeta)=\delta(e^\zeta)$, noticing that
$$
f^\prime(\zeta)=e^\zeta h^\prime(e^\zeta)=\xi\partial_\xi
h(\xi)
$$
and
$$ |f(\zeta)|\le\frac C{(1+|e^\zeta|^{2k})d(\zeta)^n}.
$$
\\[-5pt]
\end{proof}

The estimate \eqref{equation:hnkC} can be also expressed as follows: for all $\xi\in S$,
\begin{eqnarray*}
&&(1+|\xi|^{2k})\delta(\xi)^{n+1}|\xi \partial_\xi h(\xi)|\\
&&\qquad\qquad\qquad\le (n+1) e^{2k+1}\,\sup_{\xi'\in S}|h(\xi')(1+|\xi'|^{2k})\delta(\xi')^n|;
\end{eqnarray*}
if we put $k=1$ and replace $n$ by $n-1$, we find, for all $\xi\in S$ and $n\ge1$:
\begin{eqnarray}\label{equation:hn1}
&&(1+|\xi|^2)\delta(\xi)^{n}|\xi \partial_\xi h(\xi)|\\
&&\qquad\qquad\qquad\le  e^{3}\,n\,\sup_{\xi'\in S}|h(\xi')(1+|\xi'|^{2})\delta(\xi')^{n-1}|\,.
\nonumber
\end{eqnarray}
This estimate permits to establish the following interesting result.

\begin{cor}\label{cor:key}
Let $S=S(d,\theta)$, $\mu$ and $n$ as in Definition~\ref{definition:Nagumo}. Suppose $n\ge 1$ and let $f\in\EE_{S,\mu,n-1}$. If $\xi f\in\EE_{S,\mu,n}$, then $\xi\partial_\xi f(\xi)\in\EE_{S,\mu,n}$ and, moreover, the following estimate holds:
\begin{equation}\label{equation:key1}
 \Vert\xi\partial_\xi f(\xi)\Vert_{S,\mu,n}\le e^3\,n\Vert f\Vert_{S,\mu,n-1}+\vert\mu\vert\Vert\xi f\Vert_{S,\mu,n}\,.
\end{equation}
\end{cor}

\begin{proof}
If we write $h(\xi)= f(\xi)e^{-\mu\xi}$ for all $\xi\in S$, it follows that
$$
M_0\vert (1+\xi|^2)\vert h(\xi)\vert \delta(\xi)^{n-1}\le \Vert f\Vert_{S,\mu,n-1},
$$
where $M_0$ denotes the positive constant given by \eqref{equation:M_0}.
From relation \eqref{equation:hn1} one deduces immediately that
\begin{equation}\label{equation:hfn-1}
M_0(1+|\xi|^2)\delta(\xi)^{n}|\xi \partial_\xi h(\xi)|\le  {e^3}\,n \Vert f\Vert_{S,\mu,n-1}.
\end{equation}

On the other hand, since
\begin{equation*}\label{equation:hf}
\xi\partial_\xi h(\xi)=e^{-\mu\xi} \xi\partial_\xi f(\xi)-\mu \xi f(\xi)e^{-\mu\xi},
\end{equation*}
from \eqref{equation:hfn-1} we obtain:
\begin{eqnarray*}
&&M_0|(1+|\xi|^2)\delta^{n}(\xi)e^{-\mu \xi} \xi\partial_\xi f(\xi)|
\\
&&\qquad\qquad\qquad\le {e^3}\,n\,\Vert f\Vert_{S,\mu,n-1} + M_0(1+|\xi|^2)\delta(\xi)^n\,|e^{-\mu \xi}\mu\xi f(\xi)|\,. \nonumber
\end{eqnarray*}
We finish the proof by taking the sup of both sides for all $\xi\in S$ and making use of the definition of $\Vert\cdot\Vert_{S,\mu,n}$ and that of $\Vert\cdot\Vert_{S,\mu,n-1}$, respectively.
\end{proof}
 
\subsection{Proof of lemma~\ref{lemma:key}}\label{subsection:keyproof}
\begin{proof}By hypothesis, $(n-b-c\xi)f\in\EE_{S,\mu,n-1}$ with $n\ge 1$; so, we may define:
\begin{equation}\label{equation:Kn-1}
K_{n-1}:=\Vert (n-b-c\xi)f\Vert_{S,\mu,n-1}<\infty.
\end{equation}
As in the proof of Corollary~\ref{cor:key}, let $h(\xi)= f(\xi)e^{-\mu\xi}$ for all $\xi\in S$ and, by a similar way, we can find the following estimate:
\begin{eqnarray}\label{equation:M0h}
&&M_0|(1+|\xi|^2)\delta^{n}(\xi)e^{-\mu \xi} \xi\partial_\xi f(\xi)|
\\
&&\qquad\qquad\qquad\le \frac{e^3}\sigma\,K_{n-1} + M_0(1+|\xi|^2)\delta(\xi)^n\,|e^{-\mu \xi}\mu\xi f(\xi)|\,. \nonumber
\end{eqnarray}
Since $\delta(\xi)\le 1$, relation \eqref{equation:M0h} implies that
$$
\Vert\xi\partial_\xi f(\xi)\Vert_{S,\mu,n}\le C\,K_{n-1}=C\,\Vert (n-b-c\xi)f\Vert_{S,\mu,n-1}
$$
if we set
$$
C=\frac{e^3}\sigma+\sup_{\xi\in S}\frac{\vert\mu\xi|}{|n-b-c\xi|}.
$$
So from \eqref{equation:nbcsigma}, we have
$$
C<\frac{e^3+\vert\mu\vert}{\sigma}=E.
$$
The proof of Lemma \ref{lemma:key} is complete. 
\\[-5pt]
\end{proof}

\section{Two extensions}\label{section:extension}

The present section will be devoted to make some extensions  for results obtained in the last two sections, \S\ref{section:spaces} and \S\ref{section:lemma}. The first extension will be given by adding to any sector $S(d,\theta)$ an open disc centered at the origin, and the second one will concern  the case of any positive level $k$.

\subsection{Case of a sector joined by a disc}\label{subsection:extensiondisc}

For any $R\ge0$ and $d\in\RR$, $\theta\in(0,\pi)$, we define
$$
S(R;d,\theta):=S(d,\theta)\cup\{\xi\in\CC: 0<\vert \xi\vert<R\}\,.
$$
Noticing that $S(0;d,\theta)=S(d,\theta)$, we will see how to continue to have results known for $S(d,\theta)$ while replaced by $S(R;d,\theta)$. A such sector may be said {\it sector joined by a disc}.

In the proof of Proposition~\ref{proposition:Shh'}, one identifies each open sector
$S(d,\theta)$ to a horizontal strip, saying
$$\Omega(d,\theta)=\RR \times \{(d-\theta,d+\theta)i\},
$$ via the complex logarithm application $\log$ (with principal branch...) and, by this way, the angular
distance $\delta(\xi,S)$ given by \eqref{equation:deltaxi}
is exactly the distance of $\log \xi$ to the boundary of
$\Omega(d,\theta)$. This observation inspires the following definition.

\begin{definition}\label{definition:deltaSR}
Let $S=S(R;d,\theta)$ with $R>0$. Let
$$\Omega=\Omega(R;d,\theta):=\Omega(d,\theta)\cup\{\eta\in\CC:\Re\eta<\ln R\}.
$$
We define, for any $\xi\in S$:
\begin{equation}\label{equation:deltaSR}
 \delta(\xi))=\delta(\xi,S):=\min(1,\inf_{\eta\in\partial\Omega}\vert
\log\xi-\eta\vert)\,,
\end{equation}
where $\log\xi$ denotes any number $\eta\in\Omega(R;d,\theta)$ such that $\xi=e^\eta$.
\end{definition}

See Figure A below for the correspondence between $S(R;d,\theta)$ and $\Omega(R;d,\theta)$. As $R\to0$, we see the domain $\Omega(R;d,\theta)$ approaching the horizontal strip $\Omega(d,\theta)$.

\begin{figure}[htbp]
 \centerline{\includegraphics[width=3.4in,
height=1.3in]{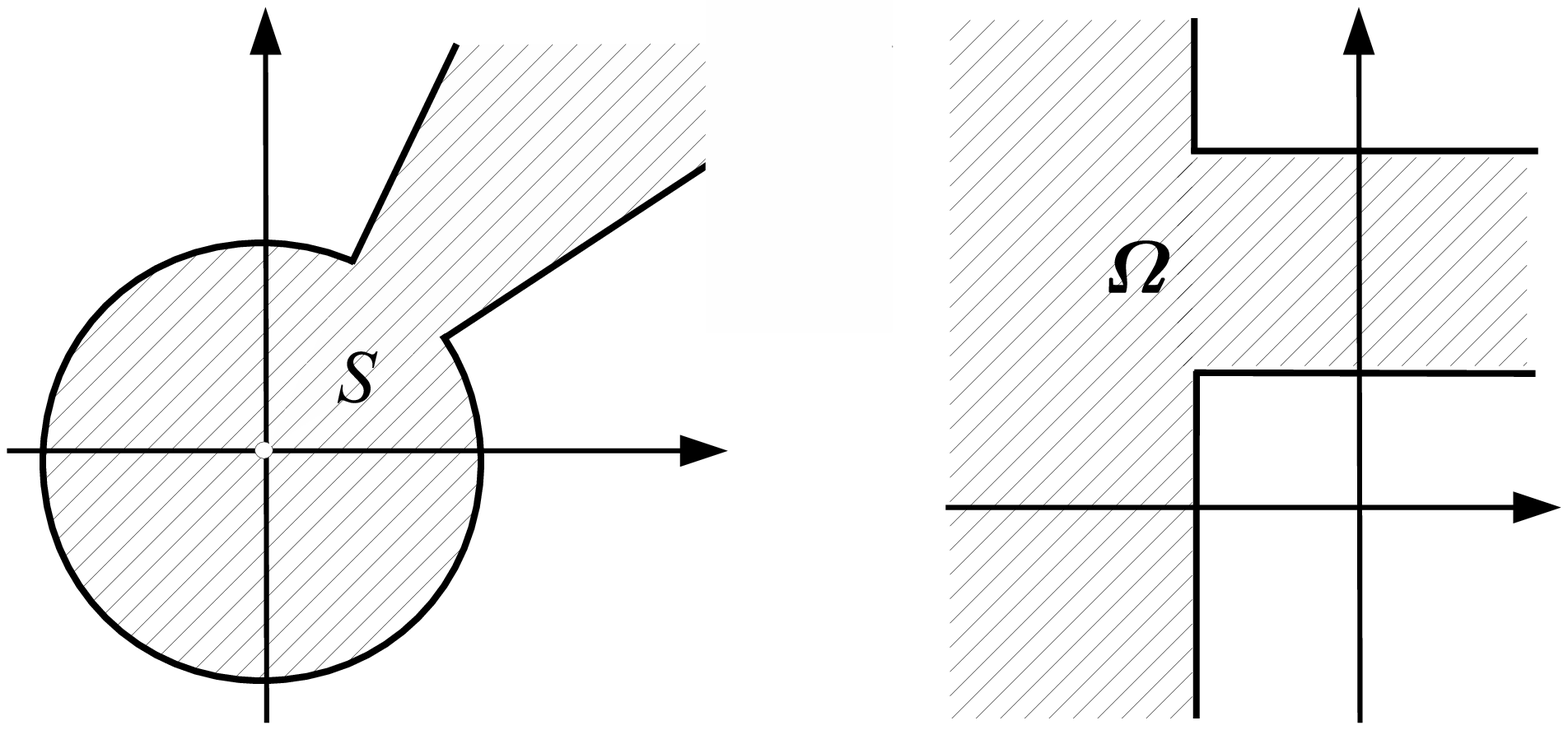}}
$$S(R;d,\theta)\ni\xi\longrightarrow \eta=\log\xi\in\Omega(R;d,\theta)$$
\vskip 2mm
\centerline{\bf Figure A}
\end{figure}

Remark that if $\xi$ belongs to the open disc $D(0;R)$, then $\delta(\xi,S)$ depends on the distance of $\xi$ to the boundary of the disc, that is to say, $\delta(\xi,S)$ depends of $\vert\xi\vert$. In this case, \eqref{equation:deltaSR} can be read as follows:
\begin{equation}\label{equation:deltaR}
\delta(\xi)=1\quad\hbox{or}\quad
 \delta(\xi)=\ln R-\ln\vert\xi\vert
\end{equation}
if
$$
0<\vert\xi\vert\le R/e\quad\hbox{or}\quad
R/e\le \vert\xi\vert<R,
$$
respectively.

\begin{definition}\label{definition:NagumoR}
Let $S=S(R;d,\theta)$, $\mu\in(0,\infty e^{-id})$ and let
$n\ge 0$.  Let $f\in\OO(S)$. We say that $f$ belongs to $\EE_{S,\mu,n}$ if
$$
\|f\|_{S,\mu ,n}:=M_0\,\sup_{\xi\in S}|f(\xi)e^{-\mu
 \xi}(1+|\xi|^2)\delta(\xi)^n |<\infty,
$$
where $M_0$ denotes the positive constant given in \eqref{equation:M_0} and $\delta(\xi)$, the function defined by \eqref{equation:deltaSR}.
\end{definition}

In the above, the set $S(R;d,\theta)$ does not contain the point at the origin of the complex plane and, therefore, the function $f$ is not, {\it a priori}, assumed to be defined at this point. From \eqref{equation:deltaR}, one may observe that, as $\xi\to0$,
$$
\vert f(\xi)\vert\lesssim\frac{\|f\|_{S,\mu ,n}}{M_0},
$$
which implies that $f$ can be continued
to be an analytic function at $\xi=0$.

\begin{proposition}\label{proposition:fnear0}
Let $S=S(R;d,\theta)$, with $R>0$. Let $f\in\OO(S)$. If $\lim_{\xi\to0}f(\xi)=0$ and $f\in\EE_{S,\mu,n}$, then $\frac f\xi\in\EE_{S,\mu,n}$ and, moreover:
\begin{equation}\label{equation:f/xi}
\Vert \frac f\xi\Vert_{S,\mu,n}\le \frac eR\,\Vert f\Vert_{S,\mu,n}.
\end{equation}
\end{proposition}

\begin{proof}
In view of the fact that $f$ may be analytically continued at zero and that its limit value is zero, it follows that $\frac f\xi$ can be continued as analytic function at $\xi=0$. Therefore, applying the maximum principle to $\frac f\xi$ on the closed disc $\bar D:=\bar D(0; R/e)$ allows us to get \eqref{equation:f/xi} if one checks the definition of $\Vert\frac f\xi\Vert_{S,\mu,n}$ over the disc $\bar D$ and then over its complement separately. We omit the details.
\end{proof}

One can state similar properties for Banach spaces
$(\EE_{S,\mu,n},\|\cdot\|)$ as in the case of $R=0$. Namely, instead
of Proposition~\ref{proposition:convolution}, one can notice the following fact.
\begin{proposition}\label{proposition:NagumoR}
Let $S=S(R;d,\theta)$, with $d\in\SS^1$, $\theta\in(0,\pi)$ and
$R>0$.
If $f\in\EE_{S,\mu ,n}$ and $g\in\EE_{S,\mu
,n^\prime}$, then
 $f*g\in\EE_{S,\mu ,n+n^\prime}$
 and $$\|f*g\|_{S,\mu ,n+n^\prime}\le \|f\|_{S,\mu ,n}\,\|g\|_{S,\mu ,n^\prime}.$$
\end{proposition}

\begin{proof}
Let $\xi\in S$. For any $\tau\in (0,\xi)$, one can see that
\begin{equation}\label{equation:deltatau}
\delta(\tau,S)\ge \delta(\xi,S),\quad
\delta(\xi-\tau,S)\ge \delta(\xi,S),
\end{equation}
so that one can give a similar proof as that done for Proposition~\ref{proposition:convolution}, (1). The details are left to the reader.
\end{proof}

With respect to the derivative of a function belonging to $\EE_{S,\mu,n}$, we mention the following result.

\begin{proposition}\label{proposition:keyRderivative}
Let $S=S(R;d,\theta)$ with $R>0$. Let $f\in\EE_{S,\mu,n-1}$ for some $\mu\in(0,\infty e^{-id})$ and $n\ge1$. Then $\partial_\xi f\in\EE_{S,\mu,n}$ and
\begin{equation}\label{equation:keyRderivative}
\|\partial_\xi f(\xi)\|_{S,\mu,n} \le (\frac{ne^{4}}R+|\mu|)\, \|f(\xi)\|_{S,\mu,n-1}.
\end{equation}
\end{proposition}

\begin{proof}
As in the proof of Corollary~\ref{cor:key} and also that of Lemma~\ref{lemma:key}, we write $f(\xi)=h(\xi)e^{\mu\xi}$ and, by taking into account of \eqref{equation:f/xi}, one can easily check that
\begin{equation}\label{equation:keyRderivativefh}
\|\partial_\xi f(\xi)\|_{S,\mu,n}\le
\frac{e}{R}\,\|e^{\mu\xi}\xi\partial_\xi h(\xi)\|_{S,\mu,n} + \|\mu f(\xi)\|_{S,\mu,n}.
\end{equation}
Since $d(\zeta):=\delta(e^\zeta,S)$ satisfies condition \eqref{equation:dboundary} of Proposition~\ref{proposition:Dff'},
one can also prove that, for any  holomorphic function $h$ in $S(R;d,\theta)$, if we let
$$
C_{h,n}:=\sup_{\xi\in S}|h(\xi)(1+|\xi|^{2})\delta(\xi,S)^{n-1}|<\infty,
$$
then the following relation holds (see the proof of Proposition~\ref{proposition:Shh'}):
$$ |(1+|\xi|^{2})\delta(\xi,S)^{n}\xi \partial_\xi h(\xi)|\le n e^{3}
C_{h,n}.
$$
Therefore, one can find that
\begin{equation}\label{eqh:no}
\|e^{\mu\xi}\xi\partial_\xi h(\xi)\|_{S,\mu,n}\le n\, e^{3}\,\|e^{\mu\xi} h(\xi)\|_{S,\mu,n-1},
\end{equation}
which, together with  \eqref{equation:keyRderivativefh} implies relation \eqref{equation:keyRderivative} and thus one ends the proof of Proposition~\ref{proposition:keyRderivative}.
\end{proof}
As application, we give the following result, that is in the same line as our key Lemma~\ref{lemma:key}.

\begin{cor}\label{cor:keyR}Let $S=S(R;d,\theta)$ with $R>0$. Let $f\in\OO(S)$.
Let $P(n,\xi)$ be a sequence of functions defined and analytic over $S$. Suppose that the following condition is fulfilled:
$$
C:=\sup\limits_{(n,\xi)\in\NN\times
S}\frac{\max(n,\vert\xi\vert)}{\left|P(n,\xi)\right|}<\infty\,.
$$
If $P(n,\xi)\, f(\xi)\in\EE_{S,\mu ,n-1}$, then
 $\partial_\xi f\in \EE_{S,\mu ,n}$, $\partial_\xi (\xi\partial_\xi) f\in\EE_{S,\mu ,n+1}$ and, moreover, the following estimates hold:
\begin{equation}\label{equation:keyRderivativeP}
 \|\partial_\xi f\|_{S,\mu ,n} \le   \frac{eE_0}R\, \|P(n,\xi)\,f\|_{S,\mu ,n-1}
\end{equation}
and \begin{equation}\label{equation:keyRderivativeP1}
\|\partial_\xi \xi\partial_\xi f\|_{S,\mu ,n+1} \le \frac{ne{E_0}^2}R\,
\|P(n,\xi)\,f\|_{S,\mu ,n-1},
    \end{equation}
 where we set
$E_0=(2e^3+|\mu|)C$.
\end{cor}

\begin{proof}
Relation \eqref{equation:keyRderivativeP} follows directly from \eqref{equation:keyRderivative} and the definition of constant $C$. Furthermore, one can obtain \eqref{equation:keyRderivativeP1} from \eqref{eqh:no}, by observing that
\begin{equation*}\label{eqh2:no}
\|\xi \partial_\xi f(\xi)\|_{S,\mu,n}\le
M_0\sup\limits_{\xi\in S}|(n e^{3}+|\mu\xi|) f(\xi)e^{-\mu\xi}(1+|\xi|^2)\delta(\xi,S)^{n-1}|\,.
\end{equation*}
\\[-5pt]
\end{proof}

\subsection{Extension to an arbitrary level $k>0$}\label{subsection:extensionk}
In the rest of this section, we will discuss the case of any arbitrary positive level $k$.
Indeed, the Borel summability requires an exponential growth of at most order one at infinity where the general $k$-summability needs to make use of exponential growth of order $k$. For this matter, one can see \cite{Ra,Ba,Br,MR}.

We firstly give a version of level $k$ for the Banach spaces $(\EE_{S,\mu,n},\Vert\cdot\Vert_{S,\mu,n})$. In what follows, if $S=S(d,\theta)$, we define the so-called {\it $k$-ramified} sector $S^{(k)}$ by
$$
S^{(k)}=S^{(k)}(d,\theta):=S(\frac dk,\frac \theta k),
$$
so that we may write the following $1-1$ ramification map:
$$S(d, \theta )\ni\xi\mapsto \rho_k\xi:=\xi^{1/k}\in S^{(k)}(d,\theta).
 $$
 More general, if $f$ is a function given in some sector $S^{(k)}$, we will denote by $\rho_kf$ the function defined in the sector $S$ by the following relation:
$$
\forall \xi\in S,\quad
\rho_kf(\xi)=f(\xi^{1/k}).
$$

\begin{definition}\label{definition:kNagumo}
Let $S=S(d,\theta)$, $\mu$ and $n$ be as in Definition~\ref{definition:Nagumo}. Let $k>0$.  A function $f\in\OO(S^{(k)})$ is said belonging to the set $\EE_{S,\mu,n}^{(k)}$ if $\rho_kf\in \EE_{S,\mu,n}$. In this case, we define:
\begin{equation}\label{equation:knorm}
\Vert f\Vert^{(k)}_{S,\mu,n}=\Vert \rho_kf\Vert_{S,\mu,n}\,.
\end{equation}
\end{definition}

In other words, one may write
$\EE_{S,\mu,n}^{(k)}$ as follows:
$$
\EE_{S,\mu,n}^{(k)}={\rho_k}^{-1}(\EE_{S,\mu,n});
$$
so, it is easy to see  that each $(\EE^{(k)}_{S,\mu,n},\Vert\cdot\Vert^{(k)}_{S,\mu,n})$ constitutes a Banach space. From  \eqref{equation:2norms}, we deduce the following relations:
$$
\EE^{(k)}_{S,\mu',n'}\subset \EE^{(k)}_{S,\mu,n},\qquad
\Vert f\Vert^{(k)}_{S,\mu,n}\le \Vert f\Vert^{(k)}_{S,\mu',n'}
$$
if $n\ge n'$ and $|\mu|\ge|\mu'|$.

Let $f$ and $g$ be continuous  functions in some sector $S^{(k)}$. Following \cite[\S 2]{MR} (see also \cite[(1.7)]{Br}), the so-called {\it convolution product of level $k$ of $f$ and $g$},  traditionally denoted by $f*_kg$, is the function defined in $S^{(k)}$ by the following relation:
\begin{equation}\label{equation:fgkconvolution}
\rho_k(f*_kg)=(\rho_kf)*(\rho_kg)\,.
\end{equation}

\begin{proposition}\label{proposition:kconvolution}
Let $S=S(d,\theta)$ and $\mu$ as in Definition~\ref{definition:kNagumo} and let $n$, $n'\ge0$. Let $k>0$. The following assertions hold.
\begin{enumerate}
\item \label{assertion:kfgmunn'}If $f\in\EE_{S, \mu ,n}^{(k)}$ and $g\in\EE_{S, \mu ,n^\prime}^{(k)}$, then
 $f*_kg\in\EE_{S, \mu ,n+n^\prime}^{(k)}$
 and
\begin{equation}\label{equation:kfgmunn'}
\|f*_kg\|_{S, \mu ,n+n^\prime}^{(k)}\le \|f\|_{S, \mu ,n}^{(k)}\,\|g\|_{S, \mu ,n^\prime}^{(k)}.
\end{equation}
\item \label{assertion:kfgmumu'} Let $\mu^\prime\in(0,\infty e^{-id})$ such that
$|\mu|\le |\mu^\prime|$. If $f\in\EE_{S, \mu ,0}^{(k)}$, $g\in\EE_{S, \mu
^\prime,n}^{(k)}$, then $f*_kg\in \EE_{S, \mu ^\prime,n}^{(k)}$ and
\begin{equation}\label{equation:kfgmumu'}
\|f*_kg\|_{S, \mu ^\prime,n}^{(k)}\le C_{\mu'-\mu}\,\Vert f\Vert_{S,\mu,0}^{(k)}\,\Vert g\Vert_{S,\mu',n}^{(k)},
\end{equation}
where $
C_{\mu'-\mu}$ denotes the positive constant defined in Proposition~\ref{proposition:convolution}, Assertion~\ref{assertion:fgmumu'}.
\end{enumerate}
\end{proposition}

\begin{proof}
It follows directly from Proposition~\ref{proposition:convolution}, by taking into account the relations \eqref{equation:knorm} and \eqref{equation:fgkconvolution}.
\end{proof}

Concerning the key lemma~\ref{lemma:key}, we mention the following generalization.

\begin{proposition}\label{proposition:kkeylemma}
Let $\theta \in(0,\pi)$ and $S:=S(d,\theta)$ be such that
inequality (\ref{equation:nbcsigma}) holds, with $k>0$.
Let $n$ be a positive integer.
If $(n-b-c\xi^k)f\in\EE_{S, \mu ,n-1}^{(k)}$,
then $\xi\partial_\xi f\in\EE_{S, \mu ,n}^{(k)}$ and
\begin{equation}\label{equation:kkey}\|\xi\partial_\xi
f\|_{S, \mu ,n}^{(k)}\le \frac{k(e^3+\vert\mu\vert)}\sigma\ \|(n-b-c\xi^k)f\|_{S, \mu ,n-1}^{(k)},
\end{equation}
where $\sigma$ denotes a positive constant satisfying (\ref{equation:nbcsigma}).
\end{proposition}

\begin{proof}
It suffices to make use of Lemma~\ref{lemma:key}, by noticing the following elementary relations:
$$
\rho_k[(n-b-c\xi^k)f]=(n-b-c\xi)\rho_kf$$
and $$
\xi\partial_\xi(\rho_k f)=\frac1k\,\rho_k(\xi\partial_\xi f)\,.
$$
\\[-5pt]
\end{proof}

We leave to the reader to translate Corollary~\ref{cor:key} into the $k$-level's context.

Finally we mention that one can combine \S~\ref{subsection:extensiondisc} with \S~\ref{subsection:extensionk} to get an extension more general as follows: letting
\begin{equation}\label{equation:extensionkR}
S^{(k)}=S^{(k)}(R;d,\theta):=D(0,R^{1/k})\cup S^{(k)}(d,\theta),
\end{equation}
one may then define, by an obvious way, the functional spaces $\EE_{S,\mu,m}^{(k)}$ for $\mu \in(0,\infty e^{-id})$ and $m\ge0$.

\section{$k$-summable functions or series with holomorphic parameters}\label{section:ksummable}

In this section, we will begin by recalling some definitions concerning the $k$-summability of a power series in the sense of Ramis \cite{Ra}, including Gevrey asymptotic expansion and $k$-Borel-Laplace transformation. From \S~\ref{subsection:ksummableparameter}, we will approach $k$-summability with holomorphic parameters, this means that the fields $\CC$ of complex number can be replaced by some suitable space of holomorphic functions.

\subsection{$k$-summable series or functions and Gevrey asymptotic expansions}\label{subsection:ksummableC}
Let $d\in[0,2k\pi)$ and let $R>0$, $\epsilon>0$. We set:
\begin{equation}\label{equation:ksummableV}
V^{(k)}(R;d,\epsilon):=\Bigl\{x\in\CC:0<\vert x\vert<R,\bigl\vert\arg x-\frac dk\bigr\vert<\frac{\pi+\epsilon}{2k}\Bigr\}\,.
\end{equation}
Mind that $V^{(k)}(R;d,\epsilon)$ presents a germ of open sector at $0$ having openness strictly larger than $\pi/k$, contrary to the sector $S^{(k)}(d,\theta)$ or $S(R;d,\theta)$ or $S^{(k)}(R;d,\theta)$, that can be viewed as germ of open sectors along whole direction $d$.

\begin{definition}\label{definition:ksummableC}
A power series $\hat f:=\sum_{n\ge 0}a_nx^n\in\CC[[x]]$ is said $k$-summable in direction $d$ and will be denoted by $\hat f\in\CC\{x\}_k^d$, if one of the following equivalent conditions is satisfied:
\begin{enumerate}
 \item\label{item:ksummableCGevrey} There exist $V=V^{(k)}(R;d,\epsilon)$ and $f\in\OO(V)$ such that $f$ admits $\hat f$ as Gevrey asymptotic expansion of order $k$ at zero in $V$, in the following sense: for any $V'=V^{(k)}(R';d,\epsilon')$ with $R'\in(0,R)$ and $\epsilon'\in(0,\epsilon)$, one can find positive constants $C=C_{V'}$, $A=A_{V'}$ such that the following relation holds for any non-negative integer $N$:
\begin{equation}\label{equation:ksummableCGevrey}
\sup_{
x\in V'}\Bigl|x^{-N}\bigl(f(x)-\sum_{n=0}^{N-1}a_nx^n\bigr)\Bigr|<C\,A^N\,\Gamma(1+\frac Nk)\,.
\end{equation}
\item\label{item:ksummableCBorel} The power series
\begin{equation}\label{equation:ksummableCBorel}
\BT_k(\hat f-\sum_{n=0}^{k-1}a_nx^n):=\sum_{n\ge k}\frac{a_n}{\Gamma(1+n/k)}\xi^{n-k}
\end{equation}
defines a germ of analytic function at $\xi=0$, saying $\phi$, that can be continued in a sector $S^{(k)}(d,\theta)$ with a growth at most exponential of the first order at infinity. In other words, $\phi\in\EE^{(k)}_{S,\mu,0}$ for some suitable $S=S(R;d,\theta)$ and $\mu\in(0,\infty e^{-id})$.
\end{enumerate}

\end{definition}
If $\bigl(V^{(k)}(R_1;d,\epsilon_1);f_1\bigr)$ and $\bigl(V^{(k)}(R_2;d,\epsilon_2);f_2\bigr)$ satisfy both condition \eqref{item:ksummableCGevrey}, then $f_1=f_2$ over the intersection domain $V^{(k)}(R_1;d,\epsilon_1)\cap V^{(k)}(R_1;d,\epsilon_1)$, that inspires the following definition for the set of $k$-sums or, saying, {\it $k$-summable functions}. We denote by $\GG^{(k)}(V)\subset\OO(V)$ the set of all functions possessing a $k$-order Gevrey asymptotic expansion over $V$ and we define $\GG_d^{(k)}$ as the inductive limit of the system $\GG^{(k)}(V)$ taking for all $V=V^{(k)}(R;d,\epsilon)$, where $R>0$ and $\epsilon>0$. Therefore, for any $\hat f\in\CC\{x\}_k^d$, there exists a unique $f\in\GG_d^{(k)}$ satisfying \eqref{equation:ksummableCGevrey}, called {\it $k$-sum of $f$ in direction $d$} and is denoted by $f=\CS_k^d\hat f$.

At the same time, we write $\EE_d^{(k)}$ as the inductive limit of $\EE^{(k)}_{S,\mu,0}$ taking over all $S=S(R;d,\theta)$ with $R>0$ and $\theta>0$. The equivalence between \eqref{item:ksummableCGevrey} and \eqref{item:ksummableCBorel} can be then realized by the $k$-Borel-Laplace transform in direction $d$ (we consider only the case $a_0=a=1=...=a_{k-1}=0$):
$$
\GG^{(k)}_d\ni f\mapsto {\mathcal B}_k^df\in\EE^{(k)}_d,\quad
\EE^{(k)}_d\ni \phi\mapsto \LT_k^d\phi\in\GG_d^{(k)},
$$
where
\begin{equation}\label{equation:ksummableCBL}
{\mathcal B}_k^d={\rho_k}^{-1}\circ{\mathcal B}^d\circ\rho_k,\quad
\LT_k^d={\rho_k}^{-1}\circ\LT^d\circ\rho_k\,.
\end{equation}

\subsection{$k$-summable series or functions with holomorphic parameters}\label{subsection:ksummableparameter}
The following definition is very close to Definition 5.1.1 given in \cite[Chapitre I]{MR1}; see also \cite[\S 2.3]{To} and, for a Banach space version of $k$-summability, see \cite[Chapiter 6]{Ba}.

\begin{definition}\label{definition:ksummable}
Let $U$ be an open set of $\CC^m$, with $m\ge 1$ and let $V$ be an open sector of vertex $0$ in $x$-plane. A function $f\in\OO(U\times V)$ is said $k$-summable w.r.t. $x$ in a direction $d$ and will be denoted by $f\in\GG_d^{(k)}(\OO(U))$, if the following conditions are fulfilled:
\begin{itemize}
 \item The function $f$ can be analytically continued over $U\times V^{(k)}(R;d,\epsilon)$, for some $R>0$ and $\epsilon\in(0,\pi)$.
\item There exists a sequence $(f_n)_{n\ge 0}$ in $\OO(U)$ such that, for all relatively compact subset $U'\subset U$ and every sub-sector $V'=V^{(k)}(R';d,\epsilon')\subset V^{(k)}(R;d,\epsilon)$, one can find positive constants $C=C_{U',V'}$ and $A=A_{U',V'}$ with the following property: $\forall\,N\in\NN$,
\begin{equation}\label{equation:ksummablewithparameter}
\sup_{(z,x)\in U'\times V'}\Bigl|x^{-N}\bigl(f(z,x)-\sum_{n=0}^{N-1}f_n(z)x^n\bigr)\Bigr|<C\,A^N\,\Gamma(1+\frac Nk)\,.
\end{equation}
\end{itemize}
\end{definition}

By interpreting $\OO(U\times V)$ as being the set of analytic functions defined from $V$ to the Fr\'echet space $\OO(U)$ with the uniform norms on compacts, the above definition says that every $f\in\GG_d^{(k)}(\OO(U))$ is merely $k$-summable in direction $d$ as a function of one variable with values in $\OO(U)$. For any given $z_0\in\CC^m$, we define $\GG_d^{(k)}(\OO_{z_0})$ the set obtained by taking the inductive limit of $\GG_d^{(k)}(\OO(U))$ over all open neighborhood $U$ of $z_0$ in $\CC^m$.

The power series $\sum_{n\ge 0}f_n(z)x^n\in\OO(U)[[x]]$ satisfying \eqref{equation:ksummablewithparameter} may be called {\it $k$-order Gevrey asymptotic expansion of $f$ at $0$ in $V$ with holomorphic parameter in $U$} and it will be denoted by $\CT_x(f)$ or, in short, by $\CT(f)$ if no confusion is obvious. One can notice that if $f$ is $k$-summable in direction $d$, then it is also true for any direction sufficiently close to $d$ and that the expansion $\CT(f)$ does not depend on the choice of the direction.

On the other hand, the space $\EE_d^{(k)}$ can extend to the holomorphic parameters case as follows:  $\phi\in \EE_d^{(k)}(\OO_{z_0})$ if there exist an open neighborhood $U$ of $z_0$ in $\CC^m$, a disc plus sector $S=S(R;d,\theta)$ and $\mu\in(0,\infty e^{-id})$ such that $\phi(z,\cdot)\in \EE^{(k)}_{S,\mu,0}$, uniformly for all compact $K\subset U$:
\begin{equation}\label{equation:ksummableEE}
 \sup_{z\in K}\Vert \phi(z,\cdot)\Vert_{S,\mu,0}^{(k)}<\infty\,.
\end{equation}

We will consider only the case where $U$ is a neighborhood of $0\in\CC^m$, so that we may write $\OO_0\cong\CC\{z\}$.
\begin{definition}\label{definition:ksummableseries}
A power series
$$\hat f(z,x):=\sum_{({\bl},n)\in\NN^m\times\NN}a_{{\bf \ell},n}z^\bl x^n$$
is said $k$-summable w.r.t. $x$ in direction $d$ with holomorphic parameters at $0$ in $\CC^m$ and will be denoted by $\hat f\in \OO_0\{x\}_k^d$, if the following conditions are fulfilled:
\begin{itemize}
 \item For all $n\in\NN$, the series $\sum_{\bl\in\NN^m}a_{\bl,n}z^\bl$ defines a germ of analytic function at $0\in\NN^m$ that will be denoted by $f_n(z)$.
\item There exists $f\in \GG_d^{(k)}(\OO_0)$ such that $\CT(f)=\sum_{n\ge 0}f_n(z)x^n$.
\end{itemize}
\end{definition}

It is important to notice that in the above definition, the function $f\in \GG_d^{(k)}(\OO_0)$ is {\it unique}: it may be called {\it $k$-sum of $\hat f$ in direction $d$} and will be denoted by $f=\CS_k^d\hat f$.

In the same line as in the case of $k$-summable series with constant coefficients, we can establish the following result.

\begin{theorem}\label{theorem:ksummable}
Let $U$ be an open neighborhood of $0$ in $\CC^m$, $V=V^{(k)}(R;d,\epsilon)$ with $R>0$ and $\epsilon>0$ and let $f\in\OO(U\times V)$. We suppose that the following relation holds for all $z\in U$ and $j=0$, ..., $k-1$:
$$\lim_{V\ni x\to0}\partial_x^jf(z,x)=0\,.$$  Then the following conditions are equivalent.
\begin{enumerate}
\item We have $f\in \GG_d^{(k)}(\OO_0)$.
\item\label{item:ksummablethmphi} There exists a function $\phi\in\EE_d^{(k)}(\OO_0)$ such that $f$ can be expressed as $k$-Laplace transform of $\phi$, {\it i.e}:
    $$
    f(z,x)=(\rho_k^{-1}\circ \LT^d\circ\rho_k\phi)(z,x).
    $$
    \item\label{item:ksummablethmderivative} For all relatively compact subset $U'\subset U$ and every relatively compact sub-sector $V'=V^{(k)}(R';d,\epsilon')\subset V^{(k)}_d(R,\epsilon)$, there exist positive constants $C=C_{U',V'}$ and $A=A_{U',V'}$ such that the following relation holds for all non-negative integer $n$:
        $$
        \sup_{(z,x)\in U'\times V'}\Bigl\vert\frac{\partial_x^n f(z,x)}{n!}\Bigr\vert<CA^n\Gamma(1+\frac nk)\,.
        $$
\end{enumerate}
\end{theorem}

\begin{proof}
The proof can be done by an evident adaptation, noticing that in \eqref{item:ksummablethmphi}, $\phi$ may be obtained as the $k$-Borel transform of $\CT_x(f)$. See \cite{To,MR1} or, for the classical case of $\CC$ instead of $\OO(U)$, see \cite{Ra}.
\end{proof}

An immediate consequence is the following.

\begin{proposition}\label{proposition:ksummablederivative}\begin{enumerate}
\item \label{item:ksummablederivative1}The set $\GG_d^{(k)}(\OO_0)$ constitutes a differential algebra with respect to the usual product of functions and differential operators in $(z,x)$  and, moreover, if $f\in\GG_d^{(k)}(\OO_0)$ with $\hat f=\CT(f)$, then the following relation holds for all ${\bf \ell}\in\NN^m$ and all $n\in\NN$:
\begin{equation}\label{equation:ksummablederivative}
\CT(\partial^{\bf \ell}_z\partial_x^nf)=\partial_z^{\bf \ell}\partial_x^n\hat f.
\end{equation}
\item \label{item:summablederivative2} The set $\EE_d^{(k)}(\OO_{0})$ constitutes a differential algebra with respect to the $k$-convolution product relative to $\xi$, differential operators on $z$ and the derivative expressed by taking the product by $\xi$. Moreover, if $\phi$, $\psi\in \EE_d^{(k)}(\OO_{0})$, then the following relation holds for any $\bl\in\NN^m$:
\begin{equation}\label{equation:ksummableconvolution}
\partial_z^\bl (\phi*_{k}\psi)=\sum_{\bl_1+\bl_2=\bl}\binom{\bl_1}{\bl}\,(\partial_z^{\bl_1}\phi)*_k(\partial_z^{\bl_2}\psi).
\end{equation}

                                                          \end{enumerate}

\end{proposition}

\begin{proof}
We use of the Cauchy formula of \cite{HO} for expressing each derivative function $\partial^{\bf \ell}_z\partial_x^nf$. Then, by taking into account \eqref{item:ksummablethmderivative} of Theorem~\ref{theorem:ksummable}, we can therefore get the formula \eqref{equation:ksummablederivative}. It is similar to prove \eqref{equation:ksummableconvolution}, we omit the details here.
\end{proof}

\subsection{Taylor expansion with $k$-summable coefficients}\label{subsection:ksummableTaylor}

An element of $\GG_d^{(k)}(\OO_{0})$ can also be considered as a holomorphic function defined from a neighborhood of $0\in\CC^m$ to the Fr\'echet space $\GG_d^{(k)}$ for which the uniform norms on relatively compact sectors may be considered. The same remark remains true in the case of $\EE_d^{(k)}(\OO_0)$. So we can establish the following result.

\begin{theorem}\label{theorem:ksummableexpansion}
Let $d\in\SS^1$, $m\in\NN^*$, $k\in\NN^*$ and consider
$$\hat f:=\sum_{(\bl,n)\in\NN^m\times\NN}a_{\bl,n}z^\bl x^{n+k}\in x^k\CC[[z,x]]\,.
$$
For all $\bl\in\NN^m$, let
$$
\hat f_\bl:=\sum_{n\ge 0}a_{\bl,n}x^{n+k},\quad
\phi_\bl:=\BT_k\hat f_\bl\,.
$$
Then $\hat f\in\OO_0\{x\}_k^d$ if, and only if, one of the following equivalent conditions is satisfied:
\begin{enumerate}
 \item \label{item:ksummableexpansion1}For all $\bl\in\NN^m$,  $\hat f_\bl\in\CC\{x\}_k^d$, and  the power series
$
\sum_{\bl\in\NN^m}\CS_k^d(\hat f_\bl)z^\bl
$
is Taylor expansion of some function $f\in\GG_d^{(k)}(\OO_0)$ at $z=0\in\CC^m$. In other words, it follows that, in $\GG_d^{(k)}(\OO_0)$:
$$
\CS_k^d\hat f=\sum_{\bl\in\NN^m}\CS_k^d(\hat f_\bl)z^\bl\,.
$$
\item \label{item:ksummableexpansion2} The power series
$
\sum_{\bl\in\NN^m}\phi_\bl\, z^\bl
$
is Taylor expansion of some function $\phi\in\EE_d^{(k)}(\OO_0)$ at $z=0\in\CC^m$.\\
\item\label{item:ksummableexpansion3}There exist $R>0$, $\theta>0$, $\mu\in(0,\infty e^{-id})$, $\nu>0$ such that if $S=S(R;d,\theta)$, the power series $ \sum_{\bl\in\NN^m}\Vert\phi_\bl\Vert^{(k)}_{S,\mu,\nu|\bl|}\,z^\bl$ is Taylor expansion of some function $\phi\in\OO_0$, where $|\bl|=\ell_1+...+\ell_m$ for $\bl=(\ell_1,\cdots,\ell_m)\in\NN^m$.
\end{enumerate}
\end{theorem}

\begin{proof}
If $\hat f\in\OO_0\{x\}_k^d$ and $\hat f=\CT f$, one can express $\partial^\bl_zf(0)$ by Cauchy formula; in view of \eqref{equation:ksummablewithparameter}, we obtain that $\hat f_\bl\in\CC\{x\}^d_k$ and, furthermore, $f_\bl=\CS_k^d\hat f_\bl$, which implies the above condition \eqref{item:ksummableexpansion1}.

One can get the second condition from the first one by making use of formal $k$-Borel transform w.r.t. $x$ in the Taylor expansion of $\CS_k^d\hat f$.

Condition \eqref{item:ksummableexpansion3} can be deduced from \eqref{item:ksummableexpansion2} by merely noticing the fact that, for any $m\ge0$,
$\Vert \phi_\bl\Vert_{S,\mu,m}\le\Vert\phi_\bl\Vert_{S,\mu,0}$.

By assuming condition \eqref{item:ksummableexpansion3} and by replacing $S$ by a more smaller sector $S'=S(R';d,\theta')$, one may suppose that $ \sum_{\bl\in\NN^m}\Vert\phi_\bl\Vert^{(k)}_{S',\mu,0}\,z^\bl\in\CC\{z\}$. Therefore, applying $k$-Laplace transform yields the $k$-sum of $\hat f$, which ends the proof of Theorem~\ref{theorem:ksummableexpansion}.
\end{proof}

\begin{remark}[{\bf Convention for notations}]\label{remark:ksummablenotation}
In Part~\ref{part:summability}, instead of $z$ we will write $t$, so that the set $\OO_0$ will be merely $\CC\{t\}$. When $k=1$, we will remove the index $(k)$ or $k$ from all spaces considered above, e.g., we will write $\OO_0\{x\}^d$, $\GG_d(\OO_0)$, $\EE_d(\OO_0)$ instead of $\OO_0\{x\}^d_1$, $\GG_d^{(1)}(\OO_0)$, $\EE_d^{(1)}(\OO_0)$, respectively.
\end{remark}

\bigskip

\bigskip

\part{Summability of formal solutions of singular partial differential equations}\label{part:summability}

Let us consider the Cauchy problem \eqref{equation:1}, that is introduced in the beginning of the paper as follows:
$$\leqno(\ref{equation:1})\begin{array}{c}
t\partial_t u=a(x)t+b(x)u+x^{k+1} c(x)\partial_xu\\
[0.3cm]
\qquad\qquad\qquad\qquad+\sum\limits_{i+j+\alpha\ge
2}a_{i,j,\alpha}(x)t^iu^j(\partial_xu)^\alpha,\quad u(0,x)=0,\nonumber
\end{array}
$$
where we suppose that $a(x)$, $b(x)$, $c(x)$, $a_{i,j,\alpha}(x)$ are  holomorphic  at $x=0\in\CC$ and that $c(0)\not=0$ and $k\ge 1$.

In this part, we shall use the results of Part 1 to study the problem \eqref{equation:1}. In Section~\ref{section:conditionF}, it will be shown that for any equation \eqref{equation:1} 
with the condition $(F)$ can be regarded to have such form that the term $\partial_xu$ appears 
always as $x\partial_xu$; see the equation \eqref{equation:vu} below. This preparative form will be used in Sections ~\ref{section:proofk=1} and~\ref{section:proofgeneralcase}, for the proof of Theorem~\ref{theorem:yes}.

In Section~\ref{section:withoutF}, the condition $(F)$ will be not necessarily satisfied and some transformations will be undertook to be able to apply Theorem~\ref{theorem:yes} or its generalization Theorem~\ref{theorem:yeskcoeff}. In particular, such transformation can be chosen to be analytic for semilinear cases; see Theorem~\ref{theorem:withoutFlinear}. Section~\ref{section:proofno} is devoted to a particular study of nonlinear equation \eqref{equation:no} in which the condition $(F)$ will be not satisfied.

\section{Analytical equivalence under condition $(F)$}\label{section:conditionF}

Remember that condition $(F)$ requires the following property:
$$
b(0)\notin \NN^*=\{1,2,3,...\}\quad\hbox{and}\quad a_{i,j,\alpha}(0)=0,\ \forall\ \alpha>0.\leqno(F)
$$
If $b(0)\in\NN^*$, this is often called {\it  resonance case}, and the equation \eqref{equation:1} may have no power series solution. So, {\bf we will always assume that $b(0)\notin\NN^*$.} In this case, it is easy to check that the equation \eqref{equation:1} admits a unique power series solution that one can put in the following form:
\begin{equation}\label{equation:utx}
\hat u(t,x):=u_{0}(t)+u_{1}(t)x+u_{2}(t)x^2+\cdots\,,
\end{equation}
where, according to \cite[Corollary 2.2]{CLT}, the coefficients functions $u_{n}$, $n=0$, $1$, $2$, $\cdots $, are all analytic in some open disc centered at $t=0$ in $t$-plane. Since $\hat u(0,x)=0$, it follows that $u_n(0)=0$ for all $n\ge 0$

\begin{proposition}\label{proposition:conditionF}
Consider the Cauchy problem \eqref{equation:1}, with $k\ge 1$. If the condition $(F)$ is satisfied, then there exists a function $v(t,x)$ holomorphic at $(0,0)\in\CC^2$ such that if the solution $u$ is replaced  by $v+xu$, the equation \eqref{equation:1} can be rewritten as following form:
\begin{equation}
\begin{array}{c}\label{equation:vu}
t\partial_t u=\tilde a(x)t+\tilde b(x)u+\tilde c(x)x^{k+1}\partial_xu\\
[0.3cm]
\qquad\qquad\qquad\qquad+\sum\limits_{i+j+\alpha\ge
2}\tilde a_{i,j,\alpha}(x)t^iu^j(x\partial_xu)^\alpha,\quad u(0,x)=0,
\end{array}
\end{equation}
where $\tilde a$, $\tilde b$, $\tilde c$ and $\tilde a_{i,j,\alpha}$ are all holomorphic at $0\in\CC$,  
\begin{equation}\label{equation:bc}
 \tilde b(0)=b(0),\quad \tilde c(0)=c(0),\end{equation}
and
\begin{equation}\label{equation:valuationa}
 \val(\tilde a)\ge k, \quad
\val(\tilde a_{i,0,0})\ge k\, \;\;\mbox{for any }i\ge 2.
\end{equation}
\end{proposition}

\begin{proof}
Assume that $a_{i,j,\alpha}(0)=0$ for all $\alpha>0$ and let $\hat u(t,x)$ be the power series solution of \eqref{equation:1} given in \eqref{equation:utx}.
Let $\ell\geq 1$ as a integer and set
\begin{equation}\label{equation:vell}
v(t,x)=u_0(t)+u_1(t)x+u_2(t)x^2+...+u_\ell(t)x^\ell,
\end{equation}
that is clearly holomorphic at $0\in\CC^2$ and $v(0,x)=0$.
If we make the change of unknown function $u=v+xw$ in \eqref{equation:1},
then by a direct computation, we can find that $w$ satisfies the following partial differential equation:
\begin{eqnarray*}\label{equation:uvw}
&&xt\partial_tw=a_1(x,t)+b_1(x)w+x^{k+2}c(x)\partial_xw\\
&&\quad\quad\quad  +R(t,x,v+xw,\partial_xv+w+x\partial_xw),\qquad
w(0,x)=0,\nonumber
\end{eqnarray*}
where
$$
R(t,x,X,Y)=\sum_{i+j+\alpha\ge 2}a_{i,j,\alpha}(x)t^iX^jY^\alpha,
$$
$$
a_1(x,t)=a(x)t+b(x)v+x^2c(x)\partial_xv-t\partial_tv
$$
and
$$
b_1(x)=x\bigl(b(x)+x^kc(x)\bigr).
$$
From \eqref{equation:vell} and the fact that $u_0(t)+u_1(t)x+u_2(t)x^2+...$ satisfies the equation \eqref{equation:1} terms by terms, one can easily see that
$$
a_1(x,t)+R(t,x,v,\partial_xv)= x^{\ell+1}tg(t,x),\quad g\in\CC\{t,x\}.
$$
Expanding the difference
$$
R(t,x,v+xw,\partial_xv+w+x\partial_xw)-R(t,x,v,\partial_xv)
,$$
one conclude the proof by choosing the integer $\ell\ge k$ and setting that
$$
\tilde a(x)=x^\ell g(0,x),\quad
\tilde b(x)=b(x)+x^kc(x).
$$
\end{proof}

From Proposition~\ref{proposition:conditionF}, one has following remark.

\begin{remark}\label{remark:conditionF}
 If the condition $(F)$ is satisfied, then the equation \eqref{equation:1} is analytically equivalent to following equation of the form:
\begin{equation}
\begin{array}{c}\label{equation:vu2}
t\partial_t u=\tilde a(x)t+\tilde b(x)u+\tilde c(x)x^{k+1}\partial_xu\\
[0.3cm]
\qquad\qquad\qquad\qquad+\sum\limits_{i+j+\alpha\ge
2}\tilde a_{i,j,\alpha}(x)t^iu^j(\partial_xu)^\alpha,\quad u(0,x)=0,
\end{array}
\end{equation}
where $\tilde a$, $\tilde b$, $\tilde c$ and $\tilde a_{i,j,\alpha}$ are all holomorphic at $0\in\CC$ and
\begin{equation*}\label{equation:bcalpha}
 \tilde b(0)=b(0),\quad \tilde c(0)=c(0),\quad
\val(\tilde a_{i,j,\alpha})\ge \alpha,\ \forall\alpha>0.
\end{equation*}

\end{remark}

Observe that, the same result holds even if the condition $(F)$ is weakened as following condition:  
$$
b(0)\notin \NN^*\quad\hbox{and}\quad val(a_{i,j,\alpha}(0))+jq>0,\ \forall\ \alpha>0,\leqno(F')
$$
where $q=\min\{\val(a_{i,0,0}): i\ge 2\}$. Indeed, if the condition $(F')$ is fulfilled and $q>0$, then there is no constant term in the power series expansion of the formal solution $\hat u$ of \eqref{equation:1} in the variable $x$. One may therefore write $u(t,x)=x w(t,x)$ in \eqref{equation:1} and deduce easily that the condition $(F)$ is then satisfied for the obtained equation on $w$; so one can get an equation of type \eqref{equation:vu2} by using the required analytical transformation.

Comparing with the condition $(F)$, the condition $(F')$ holds notably in any case where we only need the conditions $b(0)\notin \NN^*$ and $\val(a_{i,0,\alpha})>0$  for the equation \eqref{equation:1}.

\subsection{Extension to case where coefficients are given in a sector}\label{subsection:sectorF} In this subsection, we shall consider the equation \eqref{equation:1} again with the right hand side function $F(t,x,u,v)$ to be only analytic in $D\times V\times D\times D$, where $D$ is an open disc centered at $0\in\CC$ and $V=\{x\in\CC:\theta_1<\arg x<\theta_2,0<\vert x\vert<R\}$ is a germ of open sector of vertex at $0\in\CC$. Also here we suppose that the function $F(t,x,u,v)$ admits an asymptotic expansion for $x\to 0$ in $V$, i.e. there exists a sequence $(F_n)_{n\in\NN}$ of elements of $\OO(D^3)$ such that
$$
F(t,x,u,v)-\sum_{n=0}^{N-1}F_n(t,u,v)x^n=O(x^{N}),\quad
\forall\,N\in\NN\,;
$$ 
(cf. \cite[\S 1]{MR1} for the definition of an asymptotic expansion with holomorphic parameters). In this case, the function $F$ can be expanded again as follows:
\begin{equation}\label{equation:expansionF}
F(t,x,u,v)= a(x)t+b(x)u+\gamma(x)v+\sum_{i+j+\alpha\ge 2}a_{i,j,\alpha}(x)t^iu^jv^\alpha\,,
\end{equation}
where for the functions $a(x)$, $b(x)$, $\gamma(x)$, $a_{i,j,\alpha}(x)$, all belonging to $\OO(V)$, each of them has an asymptotic expansion as $x\to 0$ in $V$. In order to interpret the condition $(F)$ in this case, we adopt the following natural extension of the valuation at $0\in\CC$ for an element $f\in\OO(V)$ : if $f$ admits an asymptotic expansion $f_0+f_1x+f_2x^2+...$ for $x\to 0$ in $V$, then:
$$\val (f)=\sup\{n\in\NN:f_0=...=f_{n-1}=0,f_n\not=0\}\,.$$ We can therefore notice that $\val(f)=\infty$ if and only if $f$ is infinitely flat as $x\to 0$ in $V$.

\begin{remark}\label{remark:conditionF2}
Let $F$ be a function given as in \eqref{equation:expansionF} and let
$$k=\val(\gamma)-1\ge 1,\quad c(x)=\gamma(x)/x^{k+1}.
$$
If $\displaystyle\lim_{x\to0}b(x)\notin\NN^*$ and $\val(a_{i,j,\alpha})>0$ for all $\alpha>0$, then there exists a function $v(t,x)$ which is holomorphic at $(0,0)\in\CC^2$ such that under the transformation of the unknown function $u$ by $v+xu$, then the Cauchy problem
\begin{equation}\label{equation:problemF}
t\partial_t u=F(t,x,u,\partial_xu),\quad
u(0,x)=0\,
\end{equation}
can be deduced into the form \eqref{equation:vu}. Here we use the same notations as that in Proposition~\ref{proposition:conditionF}, where the functions $\tilde a$, $\tilde b$, $\tilde c$ and $\tilde a_{i,j,\alpha}$ are all holomorphic in $V$ and possess each an asymptotic expansion at $0$, and condition \eqref{equation:bc} may be read as follows: as $x\to 0$ in $V$,
$$\lim_{x\to0}\tilde b(x)=\lim_{x\to0}b(x),\quad \lim_{x\to0}\tilde c(x)=\lim_{x\to0}c(x)\,.
$$
\end{remark}

Indeed, the proof of Proposition~\ref{proposition:conditionF} may be easily adapted, by considering the following fact: The equation \eqref{equation:problemF} has a formal solution $\hat u\in\CC[[t,x]]$ and, moreover, one can prove that $\hat u\in t\CC\{t\}[[x]]$.

The situation of Remark~\ref{remark:conditionF2} will be discussed in subsection~\ref{subsection:coefficientsk} for the summability of the solutions.

\section{Proof of Theorem~\ref{theorem:yes} in case of $k=1$}\label{section:proofk=1}

In case of $k=1$, the corresponding $k$-summability becomes the classical Borel summability. See Remark~\ref{remark:ksummablenotation} for convention of notations.

Observe the Borel summability of a power series solution of any analytic ODE or PDE may be obtained by studying, in the Borel plane, the convolution functional equation obtained from the given equation. We will apply this idea to the Cauchy problem \eqref{equation:1} and then to prove, for every suitable direction $d$, the existence of solution in $\EE_d(\OO_0)$ for the transformed equation; see the equation \eqref{equation:utx*} below.

Let us assume \eqref{equation:1} to be given in the form \eqref{equation:vu} with conditions  \eqref{equation:bc} and \eqref{equation:valuationa}, so that the power series solution in \eqref{equation:utx} starts from the first order term $u_1(t)$ w.r.t. $x$, that means
\begin{equation}\label{equation:utx1}
\hat u(t,x)=u_1(t)x+u_2(t)x^2+\cdots+u_{n+1}(t)x^{n+1}+\cdots\,.
\end{equation}
In order to simplify the notations, instead of $\tilde a(x)$, $\tilde b(x)$, $\tilde c(x)$, $\tilde a_{i,j,k}(x)$, we will write $a(x)$, $b(x)$, $c(x)$ and $a_{i,j,k}(x)$.

Let $\tilde u(t,\xi)=\BT (u)(t,\xi)$ be the formal Borel transform with respect to $x$ of the power series solution $\hat u(t,x)$ in \eqref{equation:utx1}:
\begin{equation}\label{equation:BTu}
\tilde u(t,\xi)=u_1(t)+\frac{u_2(t)}{1!}\xi+\cdots+\frac{u_{n+1}(t)}{n!}\xi^n+\cdots\,.
\end{equation}
According to Theorem~\ref{theorem:ksummableexpansion}, we can reformulate Theorem~\ref{theorem:yes} by the following statement, where $SD_{b,c;1}$ will denote the set given by \eqref{equation:SDbck} for $k=1$.

\begin{theorem}\label{theorem:yes1}
For any direction $d\in\SS^1$ that does not belong to  $SD_{b,c;1}$, we have $\hat u\in\OO_0\{x\}^d$ or, equivalently, $\tilde u\in\EE_d(\OO_0)$, where $\OO_0=\CC\{t\}$.
\end{theorem}

The rest of the section will contain three subsections. In \S\ref{subsection:convolutionequation}, we will establish the convolution product differential equation which is satisfied by $\tilde u(t,\xi)$. In  \S\ref{subsection:lemmaBanach}, we apply Proposition~\ref{proposition:convolution} to get a contraction mapping and therefore the Banach fixed point theorem can be used. We will complete the proof of Theorem~\ref{theorem:yes1} in \S\ref{subsection:endproofyes}, which concludes the proof of Theorem~\ref{theorem:yes} in the case of $k=1$.

\subsection{Convolution product differential equation}\label{subsection:convolutionequation}
By making use of the following relations:
$$
\BT (f(x) g(x))(\xi)=\BT(f)(\xi)*\BT(g)(\xi), \quad \BT(x^2\partial_x u)(t,\xi)=\xi \BT(u)(t,\xi),
$$
and
$$
\BT(x\partial_x u)(t,\xi)=\partial_\xi(\xi\,\BT(u))(t,\xi)=(\xi\partial_\xi+1) \BT(u)(t,\xi),
$$
from \eqref{equation:vu} one obtains that $\tilde u(t,\xi)$ satisfies the following convolution product differential equation:
\begin{equation}\label{equation:utx*}
\begin{array}{l}
(t\partial_t-(b+c\xi))\tilde u
=A(\xi)t+B(\xi)*\tilde u+C(\xi)*(\xi\tilde u)\\[0.3cm]
\quad\quad+\sum\limits_{i+j+\alpha\geq 2}t^i
\left[A_{i,j,\alpha}(\xi)* \tilde u^{*j}* (\partial_\xi\xi\,\tilde u)^{*\alpha}
+B_{i,j,\alpha} \tilde u^{*j}* (\partial_\xi\xi\,\tilde u)^{*\alpha} \right].
\end{array}
\end{equation}
Here, $*$ denotes the convolution with respect to the variable $\xi$,
$$b=b(0),\quad
c=c(0)\quad
B_{i,j,\alpha}=a_{i,j,\alpha}(0)
$$
 and the functions  $A$, $B$,
$C$, $A_{i,j,\alpha}$ are the Borel transforms respectively to following functions:
$$a(x),\quad
b(x)-b,\quad
c(x)-c,\quad
a_{i,j,\alpha}(x)-a_{i,j,\alpha}(0)\,.
$$
By the condition \eqref{equation:valuationa}, one can notice that $B_{i,0,0}=0$.

If we write
\begin{equation}\label{equation:utx*x}
\tilde u(t,\xi)=\sum_{n\ge 1}{\tilde u_n}(\xi)t^n,
\end{equation}
then each coefficient ${\tilde u_n}(\xi)$ satisfies a functional equation of the
following form:
\begin{equation}\label{equation:utx*n}
(n-b-c\xi){\tilde u_n}(\xi)=B(\xi)*{\tilde u_n}(\xi)+C(\xi)*(\xi {\tilde u_n}(\xi))+F_n(\xi),
\end{equation}
where $F_n(\xi)$ only depends on $\tilde u_m$ and $\xi\partial_\xi \tilde u_m$ for $m\le n-1$. Notice that the function $\tilde u_1$ is merely solution to the following equation:
\begin{equation}\label{equation:utx*1}
(1-b-c\xi){\tilde u_1}(\xi)=B(\xi)*{\tilde u_1}(\xi)+C(\xi)*(\xi {\tilde u_1}(\xi))+A(\xi)\,.
\end{equation}

\subsection{Contraction mapping in Banach space}\label{subsection:lemmaBanach}

Since the function $F(t,x,u,v)$ appeared in \eqref{equation:1} is assumed to be holomorphic at $0\in\CC^4$, its Borel transform w.r.t. $x$, saying $(\BT F)(t,\xi,u,v)$, can be seen as an element of $\cap_{d\in\SS^1}\EE_d(\OO_0)$ with $\OO_0=\CC\{t,u,v\}$. Therefore, for any sector $S=S(d,\theta)$ with $\theta<\pi/2$, there exists  $\mu_0\in(0,\infty e^{-id})$
such that the following condition is satisfied:
\begin{equation}\label{equation:normseries}
A,\,B,\,C\in \EE_{S,\mu _0,0}\quad\hbox{and}\quad\sum_{i+j+\alpha\ge 2}\|A_{i,j,\alpha}\|_{S,\mu _0,0}t^i u^jv^\alpha\in\CC\{t,u,v\}\,.
\end{equation}

\begin{lemma}\label{lemma:leg2} Let $S=S(d,\theta)$, $\theta\in(0,\pi/2)$, $\sigma>0$ and $\mu_0\in (0,\infty e^{-id})$ be such that the conditions \eqref{equation:nbcsigma} (with $k=1$) and \eqref{equation:normseries} are satisfied. Let $\mu\in (0,\infty e^{-id})$ with 
\begin{equation}\label{equation:mumu0}
|\mu|=|\mu_0|+8(\sigma M_0\cos(\theta))^{-1}(\|B\|_{S,\mu_0,0}+\|C\|_{S,\mu_0,0}).
\end{equation}
If $F_n\in \EE_{S, \mu,m}$ for $m\geq 0$, then the equation (\ref{equation:utx*n}) has a unique solution
$\tilde u_n\in\EE_{S,\mu,m}$ and
\begin{equation}\label{esuf}
\|(n-b-c\xi)\tilde u_n\|_{S, \mu,m}\le 2 \|F_n\|_{S, \mu,m}.
\end{equation}
\end{lemma}

\begin{proof} Let $\varphi(\xi)=(n-b-c\xi)\tilde u_n(\xi)$, and
consider the mapping
$$
{\mathcal T}:\varphi\mapsto B(\xi)*\frac{\varphi(\xi)}{n-b-c\xi}+
C(\xi)*\frac{\xi\varphi(\xi)}{n-b-c\xi} +F_n(\xi).
$$
 
 Since $B$, $C\in\EE_{S,\mu_0,0}$ and $\EE_{S,\mu,m}\subset\EE_{S,\mu',m'}$ if $\vert\mu\vert\le\vert\mu'\vert$ and $m\le m'$ (cf. \eqref{equation:normseries} and \eqref{equation:2norms}). 
By taking into account Proposition~\ref{proposition:convolution} and following relations, both deduced from the condition \eqref{equation:nbcsigma}:
$$\frac1{\left| n-b-c\xi\right|}\le \frac1\sigma\,,\quad
\frac{\left|\xi\right|}{\left| n-b-c\xi\right|}\le \frac1\sigma\,,$$  
then the mapping ${\mathcal T}$, as defined above, is a mapping: $\EE_{S,\mu,m}\to \EE_{S,\mu,m}$. Thus, by using Proposition~\ref{proposition:convolution} and from the definition of $\mu$, one has that, for any pair $(\varphi,\psi)\in\EE_{S,\mu,m}\times \EE_{S,\mu,m},$, the following estimate:
$$
\|{\mathcal T}\varphi-{\mathcal T}\psi\|_{S,\mu,m}
\le \frac12\,\|\varphi-\psi\|_{S,\mu,m}\,.
$$
Hence, from the Banach fixed point theorem, the equation (\ref{equation:utx*n}) has unique
solution $\tilde u_n$ such that $(n-b-c\xi)\tilde u_n(\xi)\in\EE_{S,\mu,m}$.
Moreover, the successive approximation process shows that
$$
\|(n-b-c\xi)\tilde u_n(\xi)\|_{S,\mu,m}\le 2^{-1}\|(n-b-c\xi)\tilde u_n(\xi)\|_{S,\mu,m}+\|F_n\|_{S,\mu,m},
$$
which implies the inequality (\ref{esuf}) and therefore completes the proof of Lemma~\ref{lemma:leg2}.
\end{proof}

\subsection{Proof of Theorem~\ref{theorem:yes1}}\label{subsection:endproofyes}

\begin{proof} First, by induction on $n$, we can deduce that $\tilde u_n\in\EE_{S,\mu,n-1}$. In fact, applying the result of Lemma~\ref{lemma:leg2} to the equation \eqref{equation:utx*1}, one has that $\tilde u_1\in\EE_{S,\mu,0}$, and then, from the result of Lemma~\ref{lemma:key}, we have that $\partial_\xi\xi \tilde u_1\in\EE_{S,\mu,1}$.

Secondly, let
\begin{equation}\label{equation:Y1}
Y_1:=\max\{\|\tilde u_1\|_{S,\mu,0},\|\partial_\xi (\xi \tilde u_1)\|_{S,\mu,1}\}<\infty.
\end{equation}
Expanding all terms of the equation \eqref{equation:utx*} as power series of $t$ and by using the result of Proposition~\ref{proposition:convolution} several times, one can find that, in \eqref{equation:utx*n}, the function $F_n(\xi)$ satisfies the following estimates:
$$
\|F_n\|_{S,\mu,n-2}\le\sum_{i+j+\alpha\ge 2}W_{i,j,\alpha}\sum_{ i+|{\bf h}|+|{\bf m}|=n}
U_{j,{\bf h}}\,V_{\alpha,{\bf m}}\,,
$$
where ${\bf h}\in{\NN^*}^j$, ${\bf m}\in{\NN^*}^{\alpha}$, $|{\bf h}|=h_1+\cdots+h_j$, $|{\bf m}|=m_1+\cdots+m_{\alpha}$ and 
\begin{equation}\label{equation:Wijalpha}
W_{i,j,\alpha}=|B_{i,j,\alpha}|+\|A_{i,j,\alpha}\|_{S,\mu,i+j+\alpha-2}\,,
\end{equation}
$$
U_{j,{\bf h}}=\prod_{\ell =1}^j\|\tilde u_{h_\ell}\|_{S,\mu,h_\ell -1},\quad
V_{\alpha,{\bf m}}=
\prod_{l =1}^\alpha\|\partial_\xi(\xi \tilde u_{m_l})\|_{S,\mu,m_l -1}\,.
$$
Indeed, we may notice that
\begin{eqnarray*}
&&i+j+\alpha-2+(h_1-1)+\cdots+(h_j-1)+(m_1-1)+\cdots+(m_\alpha-1)\\
&=&i+j+\alpha-2+|{\bf h}|+|{\bf m}|-j-\alpha\\
&=&i+|{\bf h}|+|{\bf m}|-2,
\end{eqnarray*}
which shows that the condition for the indices $n$, $n'$ is satisfied as required in the relation $\Vert f*g\Vert_{S,\mu,n+n'}\le\Vert f\Vert_{S,\mu,n}\,\Vert g\Vert_{S,\mu,n'}$ as that of Proposition~\ref{proposition:convolution}.

At the same time, for any $n\ge 2$, since $\|\tilde u_n\|_{S,\mu,n-1}\le
\|\tilde u_n\|_{S,\mu,n-2}$, from \eqref{equation:nbcsigma} (with $k=1$) and Lemma~\ref{lemma:leg2}, we obtain:
$$
\|\tilde u_n\|_{S,\mu,n-1}\le \|\sigma^{-1}(n-b-c\xi)\tilde u_n\|_{S,\mu,n-2}
  \le 2\sigma^{-1}\|F_n\|_{S,\mu,n-2}.
$$
Therefore, combining Lemma~\ref{lemma:key} with Lemma~\ref{lemma:leg2} yields:
$$
\|\partial_\xi(\xi\tilde u_n)\|_{S,\mu,n-1}
  \le 2(E+\sigma^{-1})\|F_n\|_{S,\mu,n-2}\,.
$$

Next, let $Y(t)$ be the solution of following analytical functional equation: 
\begin{equation}\label{equation:Y}
Y=Y_1t+\frac{2}{\sigma}\sum_{i+j+\alpha\ge2}W_{i,j,\alpha}\,
t^i\,(2\sigma^{-1} Y)^j\,(2(E+\sigma^{-1})Y)^\alpha,
\end{equation}
with $Y(0)=0$, where $Y_1$ and $W_{i,j,\alpha}$ are defined by \eqref{equation:Y1} and \eqref{equation:Wijalpha}, respectively. From \eqref{equation:valuationa} and the definition of $A_{i,j,\alpha}$ and $B_{i,j,\alpha}$, it follows that $W_{i,0,0}=0$ for all $i\ge 2$. Applying the implicit function theorem to the equation \eqref{equation:Y} we can deduce that $Y(t)$ is analytic function at $t=0\in\CC$. So we rewrite $Y(t)$ as a power series: 
$Y(t):=\sum_{n\ge 1}Y_nt^n $, thus we have
\begin{equation}
\sum_{n\ge 1}\|\tilde u_n\|_{S,\mu,n-1}t^n\ll
Y_1t+\frac2\sigma\sum_{n\ge 2}\|F_n\|_{S,\mu,n-2}t^n\ll
 Y(t), 
\end{equation}
which implies that
$$\sum_{n\ge 1}\|\tilde u_n\|_{S,\mu,n-1}t^n\in\CC\{t\}\,.
$$

Since
$$
\sum_{n\ge 1}\|\tilde u_n\|_{S,\mu,n}t^n\ll\sum_{n\ge 1}\|\tilde u_n\|_{S,\mu,n-1}t^n\,,
$$
then from Theorem~\ref{theorem:ksummableexpansion} \eqref{item:ksummableexpansion3}, Theorem~\ref{theorem:yes1} is proved. 
\end{proof}

\section{Theorem~\ref{theorem:yes} and comments}\label{section:proofgeneralcase}

In this section, we will give the proof for our main result Theorem~\ref{theorem:yes} for arbitrary level $k>0$. Comparing with the situation of $k=1$, the difference here, instead of the Borel transform $\BT$, we shall use the composite transform $\BT\circ\rho_k$.

In \S\ref{subsection:coefficientsk}, Theorem \ref{theorem:yes} has been extended to the case of equations whose coefficients are assumed to be $k$-summable in suitable directions. This extension will be useful in next section while the condition $(F)$ will be not satisfied. In \S\ref{subsection:Stokes}, we will only discuss the Stokes lines, although a more complete work on Stokes phenomena sounds interesting.

\subsection{End of proof of Theorem~\ref{theorem:yes}}\label{subsection:proofk}

\begin{proof}
Let $k\ge 1$ and the equation \eqref{equation:1} has been transformed to the form of \eqref{equation:vu}. For simplifying the notations, instead of $\tilde a(x)$, $\cdots$, $\tilde a_{i,j,\alpha}(x)$, the coefficients are still denoted by $a(x)$, $\cdots$, $a_{i,j,\alpha}(x)$, and $b=b(0)$, $c=c(0)$. From the condition \eqref{equation:valuationa} we know that all the coefficients $a(x)$ and $a_{i,0,0}(x)$, $i=2$, $\cdots$, belong to $x^k\,\CC\{x\}$. It follows that the formal solution $\hat u(t,x)$ belongs to $x^kt\CC\{t\}[[x]]$.

Let  $\tilde u(t,\xi)$ be the $k$-Borel transform of  $\hat u(t,x)$ w.r.t. $x$; as before, we write
$$
\tilde u(t,\xi)=\sum_{n\ge 1}\tilde u_n(\xi)t^n\,.
$$  So, in view of Theorem~\ref{theorem:ksummableexpansion} \eqref{item:ksummableexpansion3}, we may complete the proof of Theorem~\ref{theorem:yes} by checking the following statement:

{\it For any direction $d\notin SD_{b,c;k}$, there exist $S=S(R;d,\theta)$, $R>0$, $\theta\in(0,\pi/(2k))$ and $\mu\in(0,\infty e^{-id})$ such that the following relation holds:
\begin{equation}\label{equation:yes}
\sum_{n\ge 1}\Vert \tilde u_n\Vert^{(k)}_{S,\mu,n}\,t^n\in\CC\{t\}\,.
\end{equation}}

Indeed, we may write $\tilde u(t,\xi)=\BT\circ\rho_k\hat u(u,\xi)$, where $\rho_k$ denotes the ramification operator of order $k$ introduced in \S~\ref{subsection:extensionk}. From the equation \eqref{equation:vu}, we know that  $\tilde u$ satisfies following functional equation:
\begin{equation*}\label{equation:utx*k}
\begin{array}{l}
(t\partial_t-(b+c\xi))\tilde u
=A(\xi)t+B(\xi)*\tilde u+C(\xi)*(k\xi\tilde u)\\[0.3cm]
\quad\quad+\sum\limits_{i+j+\alpha\geq 2}t^i
\left[A_{i,j,\alpha}(\xi)* \tilde u^{*j}* (k\partial_\xi\xi\,\tilde u)^{*\alpha}
+B_{i,j,\alpha} \tilde u^{*j}* (k\partial_\xi\xi\,\tilde u)^{*\alpha} \right],
\end{array}
\end{equation*}
where, similar to the equation \eqref{equation:utx*},  the functions $A$, $B$,
$C$ and $A_{i,j,\alpha}$ are obtained by applying successively $\rho_k$ and $\BT$  to each of
$a(x)$, $b(x)-b$, $c(x)-c$ and
$a_{i,j,\alpha}(x)-a_{i,j,\alpha}(0)$, respectively. Therefore, the proof given in \S\ref{subsection:endproofyes} may be easily adapted to prove \eqref{equation:yes}, which implies the proof of Theorem~\ref{theorem:yes}.
\end{proof}

\subsection{Case of $k$-summable coefficients in equation \eqref{equation:1}}\label{subsection:coefficientsk}
Let us come back to the initial value problem \eqref{equation:problemF}, where the function $F$ is only assumed to have an asymptotic expansion for $x$ approaching to zero in a sector of the complex plane. If we suppose that $F\in\OO_0\{x\}^d_k$ for some direction $d\in\SS^1$, with $\OO_0=\CC\{t,u,v\}$, then, in the expression \eqref{equation:expansionF}, the coefficients $a(x)$, $b(x)$, $\gamma(x)$, $a_{i,j,\alpha}(x)$ belong to $\GG_d^{(k)}$; see Theorem~~\ref{theorem:ksummableexpansion}~\eqref{item:ksummableexpansion1}. We may therefore assume the function $F$ to be given for any $x$ in some open sector $V^{(k)}(R;d,\epsilon)$ defined by \eqref{equation:ksummableV}, for a suitable $R>0$ and $\epsilon\in(0,\pi)$.

\begin{theorem}\label{theorem:yeskcoeff}Let $F$ be given as in \eqref{equation:expansionF} and $k$-summable w.r.t. $x$ in direction $d$ with holomorphic parameters at $(t,u,v)={\bf 0}\in\CC^3$. If $\lim_{x\to 0}b(x)=b\notin\NN^*$, $\lim_{x\to 0}\gamma(x)x^{k+1}=c$ and $d\notin SD_{b,c;k}$, then the problem \eqref{equation:problemF} admits a unique solution in $\OO_0\{x\}_k^d$.
\end{theorem}

\begin{proof}
According to Remark~\ref{remark:conditionF2} and from \eqref{equation:problemF}, we get an analytically equivalent equation of form \eqref{equation:vu}. Follow the proof of Theorem~\ref{theorem:yes}, we can obtain the $k$-summability in direction $d$ of the unique formal solution for this equation, and, by applying the $k$-Laplace transform, we can then construct a solution which satisfying the condition of Theorem~\ref{theorem:yeskcoeff}.

The uniqueness of the solution can be deduced from that of the formal solution and that of  $k$-sum function. See Theorem~\ref{theorem:ksummableexpansion} and Theorem~\ref{theorem:ksummable} here.
\end{proof}

\subsection{Singular directions and Stokes phenomenon}\label{subsection:Stokes}

In this paragraph, we only discuss the case of $k=1$, and the general case can be easily deduced by the help of the ramification operator of level $k$. For any positive integer $n$, we set:
$$\xi_n:=\frac{n-b}{c},\quad d_n=\arg\xi_n,\quad
L_{n}:=[\xi_n,\infty e^{id_n})
$$
and we consider the simply connected domain $\Omega_n$ defined by the following relation:
$$
\Omega_n:=\CC\setminus\cup_{\ell=1}^nL_n=\Omega_{n-1}\setminus L_n\,.
$$

By convention, we write:
$$\Omega_0=\CC,\quad
\Omega_\infty=\cup_{\ell\ge 1}\Omega_\ell\,.$$

By taking a determination of the complex logarithm over $\CC\setminus[0,\infty)$, all functions $\log(\xi-\xi_n)$ will be defined over $\Omega_\ell$ once $\ell\ge n$.
We notice also that $d\notin SD_{b,c;1}$ if and only if, there exists $\theta>0$ such that $S(d,\theta)\subset \Omega_\infty$.

\begin{definition}\label{definition:EOmega} Let $\Omega=\Omega_n$, $n\in\NN\cup\{\infty\}$ and $d\in\SS^1$ and let $f\in\OO(\Omega)$.
\begin{itemize}
 \item We say that $d$ is a proper direction in $\Omega$ if there exists $\theta>0$ such that $S(d,\theta)\subset\Omega$.
\item The function $f$ is said to belong to $\EE(\Omega)$ if $f\in\EE_d$ for any proper direction $d$ in $\Omega$.
\end{itemize}
\end{definition}

When $\Omega=\CC$, the set $\EE({\CC})$ is merely the space of entire functions possessing at infinity a growth of at most first order. Observe that, in \eqref{equation:utx*1}, all functions $A$, $B$ and $C$ belong to $\EE({\CC})$, one can see that the only singular point for $\tilde u_1$ may be $\xi=\xi_1$ and, by this way, one can analyze the location of singularities for other $\tilde u_n$. This idea can be realized by the help of the following lemma.

\begin{lemma}\label{lemma:singularities}
Let $n$ be a positive integer and let $B$, $C$ and $F\in\EE(\Omega_{n-1})$. Then the following convolution equation:
\begin{equation}\label{equation:xin}
 (\xi-\xi_n)\psi(\xi)=B*\psi(\xi)+C*(\xi\psi)(\xi)+F(\xi)\,,
\end{equation}
admits a unique solution $\psi$ in $\EE(\Omega_n)$ such that, for $\xi\to\xi_n$ in $\Omega_n$, $\psi$ can be written in the following form:
\begin{equation}\label{equation:psixin}
 \psi(\xi)=\frac1{\xi-\xi_n}\sum_{m\ge 0}A_m(\xi)\,(\log(\xi-\xi_n))^m\,,
\end{equation}
where $A_m\in \EE(\Omega_{n-1})$.
\end{lemma}

\begin{proof}
Since for any $(\ell,m)\in\NN^2$, we have 
$$
\xi^\ell*\xi^m=\frac{\ell!\,m!}{(\ell+m+1)!}\,\xi^{\ell+m+1}\,.
$$
It is easy to check that a unique germ of analytic function $\psi$ may be found near $\xi=0$ as the solution of the equation \eqref{equation:xin}. Moreover, one can carry out analytic continuation process at each point of $\Omega_n$ to find that the solution $\psi$ exists over the whole $\Omega_n$.

At the same time, by Lemma~\ref{lemma:leg2} we can obtain that $\psi\in\EE(\Omega_n)$. In order to get \eqref{equation:psixin}, we may apply the so-called {\it perturbation method}  to the equation \eqref{equation:xin}, more details of the proof can be found from Appendix \ref{section:perturbation}.
\end{proof}

The result of Lemma \ref{lemma:singularities} can be used to analyze Stokes phenomena, we hope to return to this problem in a future work.

\section{Some results without Condition $(F)$}\label{section:withoutF}

This section is devoted to some discussions while the condition $(F)$ is no longer satisfied. In \S~\ref{subsection:withoutFlinear}, the equation \eqref{equation:1} will be assumed to be linear in $\partial_xu$, that means that $a_{i,j,\alpha}=0$ for all $\alpha>1$. In this case, we will show that an analytic change of variables permits to reduce \eqref{equation:1} into the form of \eqref{equation:vu}, in which Theorem~\ref{theorem:yes} can be applied (cf. Theorem~\ref{theorem:withoutFlinear} here).

In \S~\ref{subsection:withoutFgeneral}, a {\it singular} transformation $(t,x)\mapsto (t/x,x)$ can be used to study more general Cauchy problem \eqref{equation:1} in which we only suppose the formal solution exists. Thanks to this change of variables, it will be shown, in Theorem~\ref{theorem:withoutFgeneral}, that the problem \eqref{equation:1} admits always a solution which is analytic in any suitable conical domain of the form $\{(t,x)\in\CC\times V^{(k)}(R;d,\epsilon): 0<|tx|<R\}$.

\subsection{Semilinear cases.}\label{subsection:withoutFlinear} Let us consider the Cauchy problem \eqref{equation:1} again with the conditions $a_{i,j,\alpha}=0$ for all $\alpha>0$ and $j+\alpha\geq 2$. Then we have following semilinear problem:
\begin{eqnarray}\label{equation:withoutFlinear}
&&t\partial_t u=a_1(t,x)t+a_2(t,x)x^{k+1}\partial_x u\\
&&\quad\quad\quad+a_3(t,x)t\partial_x u+g(t,x,u),\quad\quad u(0,x)=0,\nonumber
\end{eqnarray}
where $a_j(t,x)$, $1\le j\le 3$ and $g(t,x,u)$ are holomorphic at $0\in\CC^2$ or $0\in\CC^3$, respectively. Moreover, without loss of generality, we can suppose that $g(0,x,0)=\partial_tg(0,x,0)=0$.

Observe that, for the equation \eqref{equation:withoutFlinear}, the condition ($F$) is satisfied if and only if $a_3(t,0)=0$. Thus we have

\begin{theorem}\label{theorem:withoutFlinear}
Consider the equation \eqref{equation:withoutFlinear}, if $k\geq 1$, $a_2(0,0)\not=0$ and $a_3(t,0)\not=0$, then there exists a holomorphic function $f(t)$ at $t=0$ with $f(0)=0$, such that under the variable transformation, i.e. the variable $x$ being replaced by $x-f(t)$, the equation \eqref{equation:withoutFlinear} can be reduced into the form of the equation \eqref{equation:vu}.
\end{theorem}

\begin{proof}
Let $f$ be a solution of following nonlinear Fuchsian equation: 
\begin{equation}\label{transf}
ty^\prime (t)=a_2(t,-y(t))y^{k+1}(t)+a_3(t,-y(t))t,\qquad y (0)=0.
\end{equation}
Then according to Maillet-Malgrange Theorem \cite{Ma}, the problem \eqref{transf} has a unique analytic solution at $t=0$, thus the solution $f$ is a analytic function at $t=0$ with $f(0)=0$. If we set 
\begin{equation}\label{tran1}
z= x+f(t) \quad\hbox{and}\quad w(t,z)=u(t,z-f(t)),
\end{equation}
then we can rewrite \eqref{equation:withoutFlinear} into the following form:
\begin{eqnarray}\label{equation:withoutFlinearw}
t\partial_t w+t f'(t)\partial_z
w=\bar a_1(t,z)t+\bar a_2(t,z)(z-f(t))^{k+1}\partial_z w\\
 +\bar a_3(t,z)t
\partial_z w+g(t,z-f(t),w),\qquad w(0,z)=0,\nonumber
\end{eqnarray}
where, for $i=1$, $2$, $3$, we write $
\bar a_i(t,z)=a_i(t,z-f(t))$. Then the equation \eqref{equation:withoutFlinearw} becomes:
\begin{equation}\label{ex1new}
\left\{\begin{array}{l}
t\partial_t w=\bar a_1(t,z)\,t +\bar a_2(t,z) z^{k+1}\partial_z w+G(t,x,w,z\partial_z w),\\[0.3cm]
w(0,z)=0,
\end{array}\right.
\end{equation}
where 
\[ \begin{array}{r}
\displaystyle G(t,x,w,z\partial_z w)=\bar a_2(t,z)\,\sum_{j=1}^{k}\frac{(k+1)!}{j!(k+1-j)!} f^j(t)z^{k-j}\,(z\partial_z w)\\[0.3cm]
\displaystyle-\frac{\bar a_3(t,z)-\bar a_3(t,0)}z\,t\, (z\partial_z w)+g(t,z-f(t),w)\,.
\end{array}
\]
One can then complete the proof, by checking that \eqref{ex1new} is a particular case of the equation \eqref{equation:vu}, where $x$ and $u$ are replaced by $z$ and $w$, respectively.
\end{proof}

Since $a_3(t,0)\not=0,$, and from the equation \eqref{transf}, one has $\val(f)=\val(a_3(t,0)t)=q>0$, then the result of \cite{CLT} implies that $\hat u(t,x)\in \CC[[t,x]]_{1/(qk),1/k}$. Moreover, let
$$f(t)=\sum\limits^\infty_{m=q}f_mt^m,\qquad
p_n(t)=\sum\limits^n_{m=q}f_mt^m\quad (n\ge q)
$$
and let $\hat u(t,x)$ be the formal solution of \eqref{equation:withoutFlinear}. If we consider the {\it $n$-th modified} formal solution $\hat u(t,x-p_n(t)):=\hat w_n(t,x)$, then one can find that
$\hat w_n(t,x)\in\CC[[t,x]]_{1/kq_n,1/k}$, where $q_n=\val(f(t)-q_n(t))>n$. When $n\to \infty$, formally we have that $\hat w_n(t,x)\to \hat w(t,x)\in\CC\{t\}[[x]]_{1/k}$, thus by using the result of Theorem~\ref{theorem:yes}, $\hat w$ will be $k$-summable with holomorphic parameter at $0$ in almost all  direction of $x$-plane.

In order to obtain the analytic solution of the equation \eqref{equation:withoutFlinear}, we 
let
$$a_3(t,0)t=\beta t^q+O(q^{q+1}),\quad
\beta\not=0,\ q\in\NN^*
$$ and define:
\[
 V^{(q,\beta;k)}(R;d,\epsilon):=\left\{(t,x)\in D(0;R)\times\CC:\frac{t^q}\beta+x\in V^{(k)}(R;d,\epsilon)\right\}\,,
 \]
where $V^{(k)}(R;d,\epsilon)$ is defined by \eqref{equation:ksummableV}.

\begin{theorem}\label{theorem:withoutFlinear1}  For any direction $d\notin \{2j\pi-\arg a_2(0,0), j=0,1,2,\cdots,k-1\}$ and $R>0$, $\varepsilon>0$ sufficiently small, then the 
equation \eqref{equation:withoutFlinear} has a solution $u(t,x)$ which is analytic in the domain $V^{(q,\beta;k)}(R;d,\epsilon)$.
\end{theorem}

\begin{proof}
It follows from Theorem~\ref{theorem:withoutFlinear}. In this case, we can apply the result of   Theorem~\ref{theorem:yes} to the power series of the solution $\hat w(t,z)=\hat u(t,z-f(t))$, with $f(t)=t^q/\beta+O(t^{q+1})$. Thus one can complete the proof of Theorem~\ref{theorem:withoutFlinear1} by using the result of Theorem~\ref{theorem:ksummableexpansion}~\eqref{item:ksummableexpansion1}.
\end{proof}

We may notice that $z=0$ is the singular surface of the solution $w(t,z)$, that is to say  $x=-f(t)$ is the singular surface of solution $u(t,x)$. In fact, one can prove that $(t,-f(t),u(t,-f(t))$ is the characteristics of the equation \eqref{equation:withoutFlinear}. Namely we have following remark.

\begin{remark}
For the semilinear singular equation \eqref{equation:withoutFlinear}, the singularity at the origin propagates along the characteristics of this singular PDEs.
\end{remark}

\subsection{General cases}\label{subsection:withoutFgeneral}
Instead of holomorphic transformation \eqref{tran1}, we introduce the following {\it singular} transformation:
 \begin{equation}\label{tran2}
 \tau=\frac tx,\qquad
 w(\tau,x)=u(x\tau,x)\,,
\end{equation}
thus we have following obvious relations 
\begin{equation}\label{tranu2}
 t\partial_t  u=\tau\partial_\tau w, \qquad
\displaystyle x\partial_x u=x\partial_x w-\tau\partial_\tau w.
\end{equation}

\begin{theorem}\label{theorem:withoutFgeneral}
Under the only assumption that $b(0)\notin \NN^*$, then there is a unique formal solution $\hat u(t,x)$ for every equation \eqref{equation:1}. If we set $\hat w(\tau,x)=\hat u(\tau x,x)$, then $\hat w(\tau,x)$ is  $k$-summable with holomorphic parameter $\tau$ at $0$ in all
 directions of
the $x$-plane except at most a countable directions as those given in Theorem~\ref{theorem:yes}.
\end{theorem}

\begin{proof}
The existence and uniqueness of the  formal solution can be directly verified by the elementary computations. In fact, if one puts $\sum_{n\ge 1}\hat u_n(x)t^n$ in both sides of the equation \eqref{equation:1} and then identifies all coefficients of $t^n$ to get $\hat u_n(x)$; so, $\hat u(t,x)=\sum\limits^\infty_{n=1}\hat u_n(x)t^n$ will 
be the formal solution of the equation \eqref{equation:1}. Next, for each coefficient $\hat u_n$ which  will satisfy a ODE, thus, by induction on $n$, we can prove that for any given positive integer $n$, $\hat u_n$ is $k$-summable in all direction except at most for $n$ directions of $x$-plane. Given a direction $d\notin DS_{b,c;k}$, let $u_n\in\GG^{(k)}_d$ be the $k$-sum of $\hat u_n$; replacing $u(t,x)$ by $u_1(x)t+u_2(x)t^2+t^2u(t,x)$ may transform the equation \eqref{equation:1} into the following form:
\begin{eqnarray}\label{gen1}
&&t\partial_{t}u=a(x)t+b(x)u+c(x)x^{k+1}\partial _{x}u+h(x)t\partial_xu\\
&&\qquad\qquad+\sum_{i+j+\alpha\ge
2}a_{i,j,\alpha}(x)t^iu^j(t\partial_xu)^\alpha,\qquad u(0,x)=0,
\nonumber
\end{eqnarray}
where $a(x)$, ..., $a_{i,j,\alpha}(x)$ belong to $\GG_d^{(k)}$. Moreover, the right hand side in \eqref{gen1} can be written as $F(t,x,u,\partial_xu)$ with $F\in \GG_d^{(k)}(\OO_0)$, where $\OO_0=\CC\{t,u,\partial_xu\}$.

From the relations \eqref{tran2} and \eqref{tranu2}, the equation \eqref{gen1} becomes that
\begin{eqnarray*}\label{gen1t}
&&\tau\partial_{\tau}w=a(x)x\tau+b(x)w+c(x)(x^{k+1}\partial _{x}w-x^{k}\tau\partial _{\tau}w)\\
&&+h(x)(x\partial_xw-\tau\partial_\tau w)\tau+\sum_{i+j+\alpha\ge
2}a_{i,j,\alpha}(x)(x\tau)^iw^j(\tau x\partial_xw-\tau^2\partial_\tau w)^\alpha.
\nonumber
\end{eqnarray*}
By implicit function theorem, this equation can be rewritten as a partial differential equation such as $\tau\partial_\tau w=F(\tau,x,w,x\partial_xw)$ and then the proof of Theorem~\ref{theorem:withoutFgeneral} can be deduced directly by the result of  Theorem~\ref{theorem:yeskcoeff}.
\end{proof}

Applying the result of Theorem~\ref{theorem:ksummableexpansion} \eqref{item:ksummableexpansion2} we have following corollary, which implies Theorem~\ref{theorem:all} is true.

\begin{cor}\label{rethgen}
 If $b(0)\notin\NN^*$, then for any direction $d\notin SD_{b,c;k}$, there exists a sector $V^{(k)}(R;d,\epsilon)$ with $R>0$ and $\epsilon>0$, such that equation \eqref{equation:1} has a solution $u(t,x)$ which is analytic in the domain $\bigl\{(t,x)\in\CC\times  V^{(k)}(R;d,\epsilon):|t|<R|x|\bigr\}$. 
\end{cor}

\begin{proof}
By using $k$-Borel-Laplace transformation, one can construct an analytic solution from the formal power series $\hat w(\tau,x)$ of Theorem~\ref{theorem:withoutFgeneral}; see Theorem~\ref{theorem:ksummableexpansion} \eqref{item:ksummableexpansion2}.
\end{proof}

\section{Theorem~\ref{theorem:no} and summability in both variables}\label{section:proofno}

In the previous section, the proofs of Theorems~\ref{theorem:withoutFlinear} and~\ref{theorem:withoutFgeneral} depended on the special changes of variables, in which we can use 
the idea in the proofs of Theorem~\ref{theorem:yes} and Theorem~\ref{theorem:yeskcoeff} to get the results. In this section, we shall study a kind of different nonlinear singular equation \eqref{equation:no}, in which the condition ($F$) is not satisfied. Here we shall give the proof of Theorem~\ref{theorem:no}. 

The nonlinear singular equation \eqref{equation:no} is a quasilinear equation with anticipative factors, we shall discuss this problem in \S~\ref{subsection:proofnoformal}. The proof of Theorem~\ref{theorem:no} will be given in \S~\ref{subsection:proofno}, which depends on Maillet-Malgrange Theorem \cite{Ma} and some nonlinear Fuchsian ODE with coefficients in Gevrey class.

\subsection{Formal anticipative aspects}\label{subsection:proofnoformal}
Suppose the coefficient $a(x)$ of the equation \eqref{equation:no}, satisfying  $a(x)=a_0+a_1x+a_2x^2+...$. Also we expand the unknown function $u(t,x)$ as the form $u_0(t)x+u_1(t)x^2+...$, then from the equation \eqref{equation:no}, we have following relations (for all $n\ge 0$ and $u_{-1}(t)=0$):
\begin{equation}\label{equation:proofnoformalun}
t\partial_tu_n(t)=a_nt+(n-1)u_{n-1}(t)+t\sum_{\ell=1}^{n+1}\ell(n+2-\ell)\,u_\ell(t)\,u_{n+2-\ell}(t)\,,
\end{equation}

In some sense, this system may be called to be {\it anticipative}, that is to say, to determinate the term $u_n(t)$ we need to know the term $u_{n+1}(t)$.  

Since $u(0,x)=0$, it follows that $u_n(0)=0$ for all integer $n$; thus one can deduce from \eqref{equation:proofnoformalun} that
\begin{equation}\label{equation:proofnoformalun0}\partial_tu_n(0)=a_n+(n-1)\partial_t u_{n-1}(0),
\quad
2\partial_t^2u_n(0)=(n-1)\partial_t^2 u_{n-1}(0)
\end{equation}
and so on  $\cdots$.

\begin{proposition}\label{proposition:proofnoformal} For sequence $u_n(t)$, given in \eqref{equation:proofnoformalun}, with initial condition $u_n(0)=0$ for all integer $n$, then the following relations hold for all positive integer $m$ and all non-negative integer $n$:
\begin{equation}\label{equation:proofnoformal2m}
\partial^{2m}_tu_n(0)=0
\end{equation}
and
\begin{equation}\label{equation:proofnoformal2m+1}
\partial_t^{2m+1}u_n(0)=\sum_{\ell=1}^{n+1}\sum_{j=0}^{m-1}\ell(n+2-\ell)\binom{2j+1}{2m}\,U_{n,\ell}^{m,j}\,,
\end{equation}
where 
$$U_{n,\ell}^{m,j}=\partial_t^{2j+1}u_\ell(0)\,\partial_t^{2(m-j)-1}u_{n+2-\ell}(0).$$
\end{proposition}

\begin{proof}
From the formula \eqref{equation:proofnoformalun0}, we can deduce the formula \eqref{equation:proofnoformal2m} for $m=1$. Also, by a direct computation, we can get the proof of  \eqref{equation:proofnoformal2m+1} for $m=1$.  Next, by induction on $m$  for $m>1$, we can use the operator $\partial_t^{2m-1}$ or $\partial_t^{2m}$ on both sides of \eqref{equation:proofnoformalun}, 
which will deduce the required formulae \eqref{equation:proofnoformal2m} and \eqref{equation:proofnoformal2m+1}.
\end{proof}

From Proposition~\ref{proposition:proofnoformal}, we may notice that the formal solution $\hat u(t,x)$ belongs to the space $t\CC[[t^2,x]]$, which leads us to introduce the following transformation:
\begin{equation}\label{equation:proofnoformals}
s=t^2,\qquad
w(s,x)=tu(t,x)\,,
\end{equation}
thus equation \eqref{equation:no} becomes 
\begin{equation}\label{equation:proofnoformalw}
2s\partial_sw=a(x)s+w+x^2\partial_xw+(\partial_xw)^2,\quad w(0,x)=0\,.
\end{equation}
Furthermore, the relations \eqref{equation:proofnoformalun0} and \eqref{equation:proofnoformal2m+1} imply that, if we set $w(s,x)=w_0(s)+w_1(s)x+...$, then:
\begin{equation}\label{equation:proofnoformalwn}
\partial_sw_n(0)=a_n+(n-1)\partial_s w_{n-1}(0)
\end{equation}
and, for $m\ge 1$,
\begin{equation}\label{equation:proofnoformalwnm}
\frac{2m+1}{(m+1)!}\partial_s^{m+1}w_n(0)=\sum_{\ell=1}^{n+1}\sum_{j=0}^{m-1}\frac{\ell(n+2-\ell)}{m-j+1}\,W_{n,\ell}^{m,j}\,,
\end{equation}
where 
$$
W_{n,\ell}^{m,j}=\frac{\partial_s^{j+1}w_\ell(0)\,\partial_s^{m-j}w_{n+2-\ell}(0)}{(j+1)!\,(m-j)!}\,.
$$
By induction on $m$, one can express each term $\partial_s^{m+1}w_n(0)$ in terms of $\partial_sw_j(0)$ for $0\le j\le m+n+1$.

\begin{proposition}
Equation \eqref{equation:proofnoformalw} admits a unique formal solution $\hat w(s,x)$ and the Gevrey order of $\hat w(s,x)$ is exactly $(1,1)$. More precisely, if we set $\hat w(s,x)=\sum_{m,n\ge 0}w_{m,n}s^{m+1}\, x^n$, then $$\BT_{1,1}\hat w(s,x):=\sum_{m,n\ge 0}\frac{w_{m,n}}{m!\,n!}\,s^{m+1}\, x^n\in\CC\{s,x\}$$
and $\BT_{1,1}\hat w(s,x)$ is divergent if either $|s|>1$ or $|x|>1$ and $s\not=0$.

Consequently, if $\hat w_n(s)=\sum_{m\ge 0}w_{m,n}s^{m+1}$, then $\hat w_n\in\CC[[s]]_1$.
\end{proposition}

\begin{proof}
 A direct proof can be deduced by using the relations of \eqref{equation:proofnoformalwn},  \eqref{equation:proofnoformalwnm} and $\partial^{m+1}_sw_n(0)=(m+1)!w_{m,n}$, and the idea of \cite{CLT}.
\end{proof}

\subsection{Proof of Theorem~\ref{theorem:no}}\label{subsection:proofno} It is easy to see that
Theorem~\ref{theorem:no} is equivalent to following result:

\begin{theorem}\label{theorem:proofnow}   
Let $\hat w(s,x)=\sum_{m\ge 0}\hat v_m(x)s^{m+1}$ be the unique formal solution of equation \eqref{equation:proofnoformalw}. If we set
$$
W(\sigma,x):=\sum_{m\ge 0}\frac{\hat v_{m}(x)}{m!}\,\sigma^m\in\CC\{\sigma\}[[x]],
$$
then for all direction $d\in\SS^1\setminus\{0\}$, it follows that $W(\sigma,x)\in\CC\{x\}^d(\OO_0)$, where $\OO_0=\CC\{\sigma\}$.
\end{theorem}

\begin{proof}
Let $\hat w_n(s)$ be as given in Proposition~\ref{proposition:proofnoformal}, we may observe that
$$2s\partial_s\hat w_0=a(0)s+\hat w_0+\hat w_1^2.
 $$Therefore, replacing $w$ by $\hat w_0(s)+\hat w_1(s)x+w$ in \eqref{equation:proofnoformalw} yields that
\begin{equation}\label{equation:proofnow}
2s\partial_sw=\alpha(s,x)+w+x^2\partial_xw+2\hat w_1(s)\partial_xw+(\partial_x w)^2\,,
\end{equation}
where $\alpha(s,x)$ is defined as 
$$
\alpha(s,x)=(a(x)-a(0))s+\bigl(\hat w_1(s)-2s\partial_s\hat w_1(s)\bigr)x+\hat w_1(s)x^2\,.
$$
The formal solution of \eqref{equation:proofnow} can be expanded as follows:
$$\hat w^*(s,x)=\sum_{m\ge0}\hat v^*_m(x)s^{m+1}\,,
$$
where  $\hat v_m^*$ satisfies the following relation:
\begin{equation}\label{equation:proofnov*}
\hat v^*_m(x)=\hat v_m(x)-w_{m,0}-w_{m,1}x\in x^2\CC[[x]]_1.
\end{equation}
Let $W^*(\sigma,x)$ be Borel transform w.r.t. $s$ of $\hat w^*(s,x)$, it follows that
$$W^*(\sigma,x)=W(\sigma,x)-\BT\hat w_0(\sigma)-\BT\hat w_1(\sigma)x.
$$
Thus, one needs only to prove that $W^*(\sigma,x)\in\CC\{x\}^d(\OO_0)$ or, thanks to Theorem~\ref{theorem:ksummableexpansion}, it suffices to establish the following property: there exist $S=S(R;d,\theta)$, $\mu\in(0,\infty e^{-id})$ such that
\begin{equation}\label{equation:proofnoBT}
\sum_{m\ge 0}\Bigl\Vert \frac{\BT\hat v^*_m}{m!}\Bigr\Vert_{S,\mu,4 m}\,\sigma^m\in\CC\{\sigma\}\,.
\end{equation}

Let $\tilde w(s,\xi)=\BT w(s,\xi)$. Applying $\BT$ to both sides of \eqref{equation:proofnow}, we have  following convolution partial differential equation:
\begin{equation*}
2s\partial_s\tilde w=\tilde\alpha(s,\xi)+(1+\xi)\tilde w+2\hat w_1(s)\partial_\xi^2(\xi\tilde w)+(\partial_\xi^2\xi\tilde w)^{*2}\,,
\end{equation*}
where
$$
\tilde\alpha(s,\xi)=A(\xi)s+\hat w_1(s)-2s\partial_s\hat w_1(s)+\hat w_1(s)\xi
$$
and
$$
A=\BT(a(x)-a(0)).
$$
Equivalently, if we write $\tilde v_m=\BT(\hat v^*_{m})$ and
$$\tilde\alpha(s,\xi)=\sum_{m\ge 0}\tilde\alpha_m(\xi)s^{m+1},\quad
\PP f(\xi)=\partial_\xi^2(\xi f(\xi)),
$$ it follows that, for all  $m\ge 0$,
\begin{equation}\label{equation:proofnoconvolution}
(2m+1-\xi)\tilde v_m=\tilde \alpha_m+2\sum_{\ell=1}^{m-1}w_{\ell,1}\PP\tilde v_{m-\ell-1}
+\sum_{\ell=0}^{m-1}\PP\tilde v_\ell*\PP\tilde v_{m-\ell-1}\,.
\end{equation}

Since ${\PP}\tilde v_{\ell}=\partial_\xi\tilde v_\ell+\partial(\xi\partial_\xi\tilde v_\ell)$, it follows that
$$
\|{\PP}\tilde v_{\ell}(\xi)\|_{S,\mu,4\ell+2}\le \|\partial_\xi\xi\partial_\xi\tilde v_\ell(\xi)\|_{S,\mu,4\ell+2}+\|\partial_\xi\tilde v_\ell(\xi)\|_{S,\mu,4\ell+1}.
$$
By Corollary~\ref{cor:keyR}, one obtains that, for all $\ell\ge 0$,
\begin{equation}\label{equation:proofnovell}\|{\PP}\tilde v_{\ell}(\xi)\|_{S,\mu,4\ell+2}\le(\ell+1) K\|(2\ell+1-\xi)\tilde v_\ell(\xi)\|_{S,\mu,4\ell},
\end{equation}
where $K$ denotes a positive constant depending of $R$, $|\mu|$ and $C$ which is given by Corollary~\ref{cor:keyR} with
$$P(n,\xi)=\frac {n+1}2-\xi.
$$
If we let $W_m=\|(2m+1-\xi)\tilde v_m(\xi)\|_{S,\mu,4m}$, from equation \eqref{equation:proofnoconvolution} and inequality \eqref{equation:proofnovell}, we find:
\begin{eqnarray}\label{eqWk}
&&W_m \le \|\tilde\alpha_m\|_{S,\mu,4m}+2K\,\sum_{\ell=1}^{m-1}|w_{\ell,1}|(m-\ell)W_{m-\ell-1}\\
&&\qquad\qquad
+K^2\,\sum_{\ell=0}^{m-1}(\ell+1)\,(m-\ell)\,W_\ell\,W_{m-\ell-l},\nonumber
\end{eqnarray}

Let
$$
A(t)=\sum_{m\ge 0}\Vert\tilde\alpha_m\Vert_{S,\mu,4m}t^{m+1},\qquad
B(t)=\sum_{m\ge 0}|w_{m,1}|t^{m+1}\,
$$
and let
$$M(t):=\sum_{m\ge 0}M_mt^{m+1}
$$
be the formal solution of the following Fuchsian differential equation:
\begin{eqnarray}\label{eqM}
y(t)=A(t)+K^2\,(t\partial_ty(t))^2
+2K\,B(t)\,t\partial_t y(t)),
\end{eqnarray}
with $y(0)=0$.
Therefore, relation \eqref{eqWk} implies that the sequence $(W_m)$ is majored by  $(M_m)$.

By Proposition~\ref{proposition:proofnoformal}, we know that $A(t)$, $B(t)\in\CC[[t]]_1$. Thus the formal solution $M(t)$ of the Fuchsian equation \eqref{eqM} will belong to the same Gevrey class as that for the coefficients $A(t)$ and $B(t)$ (cf. Remark \ref{remark:proofno} and Appendix \ref{section:MM} for more details), which completes the proof of \eqref{equation:proofnoBT}. Theorem~\ref{theorem:proofnow} is proved.
\end{proof}

\begin{remark}\label{remark:proofno}
By making use of Malgrange's  approach \cite{Ma}, one can prove the following statement: Any formal solution of an algebraic differential equation with coefficients Gevrey order $\le 1/k$ is at most Gevrey order $1/k$ if the Newton polygon of the variational equation has no slope in interval $(0,k)$; see Appendix \ref{section:MM}.
\end{remark}

\appendix
\section{On Lemma~\ref{lemma:singularities}}\label{section:perturbation}

Instead of  \eqref{equation:xin}, we consider the following {\it perturbation} equation, with a (small) parameter $\epsilon$:
\begin{equation}\label{equation:xinepsilon}
 (\xi-\xi_n)\psi(\xi)=\epsilon B*\psi(\xi)+\epsilon C*(\xi\psi)(\xi)+F(\xi). 
\end{equation}

If we write the solution in the form
$$
\psi(\xi,\epsilon)=\sum_{\ell\ge 0}\psi_\ell(\xi)\,\epsilon^\ell\,,
$$
then comparing the coefficients of $\epsilon^\ell$ in both sides of the equation \eqref{equation:xinepsilon}, one has 
$$
\psi_0(\xi)=\frac{F(\xi)}{\xi-\xi_n}
$$
and, for any $\ell\ge 0$,
\begin{equation}\label{equation:psiell}
\psi_{\ell+1}(\xi)=\frac1{\xi-\xi_n}\,\bigl(B*\psi_{\ell}(\xi)+ C*(\xi\psi_{\ell})(\xi)\bigr).
\end{equation}

By induction on $\ell$, one can easily check that $\psi_\ell$ can be put of the following form:
$$
\psi_\ell(\xi)=\frac{1}{\xi-\xi_\ell}\,\bigl[\,\psi_{\ell,0}(\xi)+\psi_{\ell,1}(\xi)\,\log(\xi-\xi_n)+...+\psi_{\ell,\ell}(\xi)\,\log^\ell(\xi-\xi_n) \,\bigr]\,,
$$
where $\psi_{\ell,j}\in\EE({\Omega_{n-1}})$. Indeed, for any $H\in\EE({\Omega_{n-1}})$ and for all $m\in\NN$, consider the function
$$
H_m:=H*\log^m(\xi-\xi_n)\,,
$$
which is clearly defined and analytic over $\Omega_{n}$ and can be continued to the universal covering $\tilde\Omega_{n-1,n}$ of $\Omega_{n-1,n}:=\Omega_{n-1}\setminus\{\xi_n\}$. Let $\gamma_n:=\gamma_{\xi_n}$ be the continuation operator, intuitively saying `monodromy operator around $\xi=\xi_n$', acting on the set $\OO(\tilde\Omega_{n-1,n})$ and such that $$\gamma_n \log(\xi-\xi_n)=\log(\xi-\xi_n)+2\pi i.
$$ It follows that, for all positive integer $m$,
\begin{equation}\label{equation:Hm}
\gamma_n H_m-H_m=\sum_{j=0}^{m-1}\binom jm (2\pi i)^{m-j}\,H_j\,.
\end{equation}
Obviously, from the fact $H_0\in\OO(\Omega_n)$, that is merely primitive function of $H$, it follows that $\gamma_n H_0=H_0$. If we set:
\begin{equation}\label{equation:Hmsum}
H_m={(2\pi i)^m}\sum_{\ell=0}^mH_{m,\ell}\,\binom{\frac{\log(\xi-\xi_n)}{2\pi i}}\ell
\end{equation}
and suppose that $\gamma_n H_{m,\ell}=H_{m,\ell}$ for all $\ell$, then \eqref{equation:Hm} implies that the coefficients $H_{m,\ell}$ are related as follows:
$$
H_{m,\ell+1}=\sum_{j=\ell}^{m-1}\binom jmH_{j,\ell}\,.
$$
In particular, one may deduce the following formula:
$$
H_{m,m}=H_{m-1,m-1}=...=H_{1,1}=H_{0,0}:=H_0.
$$
On the other side, one may prove that there exist such functions $H_{m,\ell}$, satisfying the equation \eqref{equation:Hmsum}, and to be unique in $\EE(\Omega_{n-1})$. We omit the details of the proof.

\section{On Maillet-Malgrange Theorem}\label{section:MM}

In the following, $k$ denotes a given positive number.

 Let $m\in\NN$, $z=(z_0,...,z_m)$ and let $F(x,z)\in\CC[[x,z]]$ be a power series. Let $\delta=x\frac {d\ }{dx}$ and for all $\phi\in x\CC[[x]]$, let $\Phi=(\phi,\delta\phi,...\delta^m\phi)$. We introduce following linearized operator $L_{F,\phi}$ along $\phi$  by
$$
L_{F,\phi}:=\sum_{i=0}^m{\partial_{z_i}}F(x,\Phi)\,\delta^i\in\CC[[x]][\delta]\,;
$$
therefore one can define the so-called {\it Newton polygon} ${\mathcal N}(L_{F,\phi})$ for $L_{F,\phi}$: this is the convex envelop in $[0,m]\times[0,\infty)$ of the set consisting of all the vertical half-lines starting from $(i,v_i)$ with $v_i=\val_{x=0}{\partial_{z_i}}F(x,\Phi)$, $0\le i\le m$. A differential equation on $\phi$, $F(x,\Phi)=b(x)$, is called to be {\it Fuchsian} type at $x=0$ if $v_m\le v_i$ for all $i=0$, $\cdots$, $m$ or, equivalently, if ${\mathcal N}(L_{F,\phi})\subset [0,m]\times [v_m,\infty)$.

In the meanwhile, for any $\nu\ge 0$, let $\HH_{\nu}$ be the set of  $\hat f:=\sum_{n\ge 0}a_nx^n\in\CC[[x]]_{1/k}$ such that :
$$
\Vert\hat f\Vert_{\nu}:=\sum_{n\ge 0} \vert a_n\vert\, n^\nu\, (n!)^{-1/k}<\infty,
$$
where, by convention, we denote $0^0=1$; thus one gets a Banach space $(\HH_\nu,\Vert\cdot\Vert_\nu)$. A power series $F(x,z)\in\CC[[x,z]]$ will be said to belong to $\HH_{\nu}\{z\}$ if
$$
F(x,z):=\sum_{\bl\in\NN^{m+1}}\hat f_\bl(x)z^\bl
$$
satisfies the following condition:
$$
\sum_{\bl\in\NN^{m+1}}\Vert\hat f\Vert_\nu z^\bl\in\CC\{z\}\,.
$$

Finally for all $\lambda>0$, we write $F_\lambda(x,z)=F(\lambda x,z)$. As one extension of Maillet-Malgrange Theorem \cite[Th\'eor\`eme 1.4]{Ma}, The more details of Remark \ref{remark:proofno} can be stated as follows.

\begin{proposition}\label{proposition:MM}
Let $F\in\CC[[x,z]]$ and suppose  there exists $(\nu,\lambda)\in[0,\infty)\times(0,\infty)$ such that $F_\lambda\in\HH_\nu\{z\}$. Let $\phi\in x\CC[[x]]$ be such that $F(x,\Phi)\in\CC[[x]]_{1/k}$ and $\val_{x=0}(\partial_{z_m}F(x,\Phi))<\infty$. If ${\mathcal N}(L_{F,\phi})$ have no slope with value belonging to interval $(0,k)$, then $\phi\in\CC[[x]]_{1/k}$.
\end{proposition}

\begin{proof}
It suffices to adapt the Malgrange's idea \cite{Ma} to this situation. The details are left to the interested reader.
\end{proof}

\begin{remark}\label{remark:B.1}
Malgrange's approach \cite{Ma} can be extended to $q$-difference-differential equations and ultra-metric cases, respectively (cf. \cite{Zh} and \cite{dV}). It is not difficult to think up some generalization of these results in a similar way as that in Proposition \ref{proposition:MM} above.
\end{remark}

\bigskip

\bigskip

{\bf Acknowledgements}

{\it The research supported partially by the NSFC Grants (No 10401028 and 10631020) and the 
cooperation program PICS of CNRS. The authors would like to thank J.-P. Ramis for his encouragement and support.}

\end{document}